\documentclass{amsart}

\usepackage{amssymb,latexsym,amsfonts,amsmath}
\usepackage{graphicx}
\usepackage{diagrams}
\usepackage{comment}

\newtheorem{theorem}{Theorem}[section]
\newtheorem{lemma}[theorem]{Lemma}

\newtheorem{corollary}[theorem]{Corollary}

\newtheorem{definition}[theorem]{Definition}

\newtheorem{remark}[theorem]{Remark}
\numberwithin{equation}{section}

\def\set#1{{\{#1 \}}}
\def\boxspan{\mathit{span}}
\def\params{{\mathsf{q}}}

\newcommand{\R}{{\mathbb{R}}}

\newcommand{\Ze}{{\mathbb Z}}

\newcommand{\N}{{\mathbb{N}}}

\usepackage{dsfont}

\DeclareMathOperator{\diff}{d}
\newcommand{\ra}{\rightarrow}
\newcommand{\sigalg}{\mathcal{F}}
\newcommand{\filtration}{\mathds{F}}

\newcommand{\ul}{\underline}
\newcommand{\ol}{\overline}
\newcommand{\Let}{:=}
\newcommand{\EE}{\mathds{E}}
\newcommand{\PP}{\mathds{P}}
\newcommand{\traj}[3]{#1_{#2#3}}

\begin{document}

\begin{abstract}
Stochastic switched systems are a relevant class of stochastic hybrid systems with probabilistic evolution over a continuous domain and control-dependent discrete dynamics over a finite set of modes. 
In the past few years several different techniques have been developed to assist in the stability analysis 
of stochastic switched systems. 
However, more complex and challenging objectives related to the verification of and the controller synthesis for logic specifications have not been formally investigated for this class of systems as of yet. 
With logic specifications we mean properties expressed as formulae in linear temporal logic or as automata on infinite strings.  
This paper addresses these complex objectives by constructively deriving approximately equivalent (bisimilar) symbolic models of stochastic switched systems.  
More precisely, this paper provides two different symbolic abstraction techniques: one requires state space discretization, but the other one does not require any space discretization which can be potentially more efficient than the first one when dealing with higher dimensional stochastic switched systems. Both techniques provide
finite symbolic models that are approximately bisimilar to stochastic switched systems 
under some stability assumptions on the concrete model.  
This allows formally synthesizing controllers (switching signals) that are valid for the concrete system over the finite symbolic model, 
by means of mature automata-theoretic techniques in the literature. The effectiveness of the results are illustrated by synthesizing switching signals enforcing logic specifications for two case studies including temperature control of a six-room building. 
\end{abstract}

\title[Symbolic Models for Stochastic Switched Systems: A Discretization-Free Approach]{Symbolic Models for Stochastic Switched Systems: A Discretization and a Discretization-Free Approach} 

\author[M. Zamani]{Majid Zamani$^1$} 
\author[A. Abate]{Alessandro Abate$^2$}
\author[A. Girard]{Antoine Girard$^3$} 
\address{$^1$Department of Electrical Engineering and Information Technology, Technische Universit\"at M\"unchen, 80333, Munich, Germany.}
\email{zamani@tum.de}
\urladdr{http://www.hcs.ei.tum.de}
\address{$^2$Department of Computer Science, University of Oxford, OX1 3QD, Oxford, United Kingdom.}
\email{alessandro.abate@cs.ox.ac.uk}
\urladdr{http://www.cs.ox.ac.uk/people/alessandro.abate}
\address{$^3$Laboratoire Jean Kuntzmann, Universit\'e de Grenoble, 38041 Grenoble Cedex 9, France.}
\email{antoine.girard@imag.fr}
\urladdr{https://sites.google.com/site/antoinesgirard}

\maketitle

\section{Introduction}
Stochastic hybrid systems are dynamical systems 
comprising continuous and discrete dynamics interleaved with probabilistic noise and stochastic events \cite{BL06a}.  
Because of their versatility and generality, methods for analysis and design of stochastic hybrid systems  
carry great promise in 
many safety critical applications \cite{BL06a}. Examples of such applications include power networks, automotive, finance, 
air traffic control, biology, telecommunications, and embedded systems. 
Stochastic \emph{switched} systems are a relevant subclass of stochastic hybrid systems. 
They consist of a finite (discrete) set of modes of operation, 
each of which is associated to continuous probabilistic dynamics;      
further, their discrete dynamics, in the form of mode changes, are governed by a non-probabilistic control (switching) signal.   

It is known \cite{liberzon} that switched systems can be endowed with global behaviors that are not characteristic of the behavior of any of their modes:  
for instance, global instability may arise by proper choice over time of the discrete switches between a set of stable modes.  
This is but one of the many features that makes switched systems theoretically interesting.   
With focus on \emph{stochastic} switched systems, 
despite recent progresses on basic dynamical analysis focused on stability properties \cite{debasish}, 
there are no notable results in the literature targeting more complex objectives, 
such as those dealing with verification or (controller) synthesis for logical specifications. Examples of those specifications include linear temporal logic or automata on infinite strings, 
and as such they are not amenable to classical approaches for stochastic processes. 

A promising direction to investigate these general properties is the use of \emph{symbolic models}. 
Symbolic models are abstract descriptions of the original dynamics, 
where each abstract state (or symbol) corresponds to an aggregate of states in the concrete system. 
When a finite symbolic model 
is obtained and is formally put in relationship with the original system, 
one can leverage automata-theoretic techniques for controller synthesis over the finite model \cite{MalerPnueliSifakis95} 
to automatically synthesize controllers for the original system. 
Towards this goal, 
a relevant approach is the construction of finite-state symbolic models that are \emph{bisimilar} to the original system.  
Unfortunately, the class of continuous (-time and -space) dynamical systems admitting exactly bisimilar finite-state symbolic models is quite restrictive \cite{AHLP00,LPS00} and in particular it covers mostly non-probabilistic models. The results in \cite{manuela} provide a notion of exact stochastic bisimulation for a class of stochastic hybrid systems, 
however, \cite{manuela} does not provide any abstraction algorithm, nor does it look at the synthesis problem. 
Therefore, rather than requiring exact bisimilarity, 
one can resort to \emph{approximate bisimulation} relations \cite{girard}, 
which introduce a metric between the trajectories of the abstract and the concrete models, 
and require boundedness in time of this distance. 

The construction of approximately bisimilar symbolic models has been extensively studied for 
non-probabilistic control systems, possibly affected by disturbances 
\cite{majid4,pola,pola1} and references therein, 
as well as for 
non-probabilistic switched systems \cite{girard2}.  
However, stochastic systems, particularly when endowed with hybrid dynamics, 
have only been scarcely explored.   
With focus on these models,   
a few existing results deal with abstractions of discrete-time stochastic processes \cite{AAPLS07,abate1,azuma}.  
Results for continuous-time models cover probabilistic rectangular hybrid automata \cite{sproston} 
and stochastic dynamical systems under some contractivity assumptions \cite{abate}. 
Further, the results in \cite{julius1} only {\em check} the relationship between an uncountable abstraction and a given class of stochastic hybrid systems via the notion of stochastic (bi)simulation function. However, these results do not provide any {\em construction} of approximations, nor do they deal with {\em finite} abstractions, 
and moreover appear to be computationally tractable only in the case where no input is present. 
The recent results in \cite{majid8} and \cite{majid11} investigate the construction of finite bisimilar abstractions for continuous-time stochastic control systems, without any hybrid dynamics, and randomly switched stochastic systems, respectively, such that the discrete dynamics in the latter systems are governed by a random uncontrolled signal. Finally, the recently proposed techniques in \cite{majid10} improve
the ones in \cite{majid8} by not requiring state-space discretization but only input set discretization. 
In summary, 
to the best of our knowledge there is no comprehensive work on the automatic construction of finite bisimilar abstractions for continuous-time stochastic switched systems in which the discrete dynamics are governed by a non-probabilistic control signal.  

The main contributions of this work consist in showing the existence and the construction of approximately bisimilar symbolic models for incrementally stable stochastic switched systems using two different techniques: one requires state space discretization and the other one does not require any space discretization. Note that all the techniques provided in \cite{majid4,pola,pola1,girard2,AAPLS07,abate1,azuma,sproston,abate,majid8,majid11} are only based on the discretization of state sets. Therefore, they suffer severely from \emph{the curse of dimensionality} due to gridding those sets, which is especially irritating for models with high-dimensional state sets. We also provide a simple criterion in which one can choose between the two proposed approaches the most suitable one (based on the size of the abstraction) for a given stochastic switched system. Another advantage of the second proposed approach here is that it allows one to construct symbolic models with probabilistic output values, resulting possibly in less conservative symbolic abstractions in comparison with the first proposed approach and the ones in \cite{majid8,majid11} allowing for non-probabilistic output values only. Furthermore, the second proposed approach here allows one to construct symbolic models for any given precision $\varepsilon$ and any given sampling time, but the first proposed approach and the ones in \cite{majid8,majid11} may not be applicable for a given sampling time.

Incremental stability is a property on which the main proposed results of this paper rely. This type of stability requires uniform asymptotic stability of every trajectory, rather than stability of an equilibrium point or a particular time-varying trajectory. 
In this work, we show the description of incremental stability in terms of a so-called common Lyapunov function or of multiple Lyapunov functions. 
The main results are illustrated by synthesizing controllers (switching signals) for two examples. 
First, we consider a room temperature control problem (admitting a common Lyapunov function) for a six-room building. 
We synthesize a switching signal regulating the temperature toward a desired level which is not tractable using the first proposed technique. 
The second example illustrates the use of multiple Lyapunov functions (one per mode) using the first proposed approach. A preliminary investigation on the construction of bisimilar symbolic models for stochastic switched systems using the first proposed approach (requiring state space discretization) appeared in \cite{majid9}. In this paper we present a detailed and mature description of the results presented in \cite{majid9}, including proofs, as well as proposing a second approach which does not require any space discretization.       
\vspace{-0.2cm}

\section{Stochastic Switched Systems}
\subsection{Notation} 
The symbols $\N$, $\N_0$, $\Ze$, $\R$, $\R^+$, and $\R_0^+$ denote the set of natural, nonnegative integer, integer, real, positive, and nonnegative real numbers, respectively. The symbols $I_n$, $0_n$, and $0_{n\times{m}}$ denote the identity matrix, zero vector, and zero matrix in $\R^{n\times{n}}$, $\R^n$, and $\R^{n\times{m}}$, respectively. Given a set $A$, define $A^{n+1}=A\times A^n$ for any $n\in\N$. Given a vector \mbox{$x\in\mathbb{R}^{n}$}, we denote by $x_{i}$ the $i$--th element of $x$, and by $\Vert x\Vert$ the infinity norm of $x$, namely, \mbox{$\Vert x\Vert=\max\{|x_1|,|x_2|,...,|x_n|\}$}, where $|x_i|$ denotes the absolute value of $x_i$. 
Given a matrix $P=\{p_{ij}\}\in\R^{n\times{n}}$, we denote by $\text{Tr}(P)=\sum_{i=1}^np_{ii}$ the trace of $P$.
The \emph{diagonal set} $\Delta\subset\R^n\times\R^n$ is defined as: $\Delta=\left\{(x,x) \mid x\in \R^n\right\}$.

The closed ball centered at $x\in{\mathbb{R}}^{n}$ with radius $\varepsilon$ is defined by \mbox{$\mathcal{B}_{\varepsilon}(x)=\{y\in{\mathbb{R}}^{n}\,|\,\Vert x-y\Vert\leq\varepsilon\}$}. A set $B\subseteq \R^n$ is called a 
{\em box} if $B = \prod_{i=1}^n [c_i, d_i]$, where $c_i,d_i\in \R$ with $c_i < d_i$ for each $i\in\set{1,\ldots,n}$.
The {\em span} of a box $B$ is defined as $\boxspan(B) = \min\set{ | d_i - c_i| \mid i=1,\ldots,n}$. By defining $[\R^n]_{\eta}=\left\{a\in \R^n\mid a_{i}=k_{i}\eta,k_{i}\in\mathbb{Z},i=1,\ldots,n\right\}$, the 
set \mbox{$\bigcup_{p\in[\R^n]_{\eta}}\mathcal{B}_{\lambda}(p)$} is a 
countable covering of $\R^n$ for any $\eta\in\R^+$ and $\lambda\geq\eta/2$. 
For a box $B\subseteq\R^n$ and $\eta \leq \boxspan(B)$,
define the $\eta$-approximation $[B]_\eta = [\R^n]_{\eta}\cap{B}$. 
Note that $[B]_{\eta}\neq\varnothing$ for any $\eta\leq\boxspan(B)$. 
Geometrically, for any $\eta\in{\mathbb{R}^+}$ with $\eta\leq\boxspan(B)$ and $\lambda\geq\eta$, the collection of sets 
\mbox{$\{\mathcal{B}_{\lambda}(p)\}_{p\in [B]_{\eta}}$}
is a finite covering of $B$, i.e., \mbox{$B\subseteq\bigcup_{p\in[B]_{\eta}}\mathcal{B}_{\lambda}(p)$}. 
We extend the notions of $\boxspan$ and of {\em approximation} to finite unions of boxes as follows.
Let $A = \bigcup_{j=1}^M A_j$, where each $A_j$ is a box.
Define $\boxspan(A) = \min\set{\boxspan(A_j)\mid j=1,\ldots,M}$,
and for any $\eta \leq \boxspan(A)$, define $[A]_\eta = \bigcup_{j=1}^M [A_j]_\eta$.

A continuous function \mbox{$\gamma:\mathbb{R}_{0}^{+}\rightarrow\mathbb{R}_{0}^{+}$}, is said to belong to class $\mathcal{K}$ if it is strictly increasing and \mbox{$\gamma(0)=0$}; $\gamma$ is said to belong to class $\mathcal{K}_{\infty}$ if \mbox{$\gamma\in\mathcal{K}$} and $\gamma(r)\rightarrow\infty$ as $r\rightarrow\infty$. A continuous function \mbox{$\beta:\mathbb{R}_{0}^{+}\times\mathbb{R}_{0}^{+}\rightarrow\mathbb{R}_{0}^{+}$} is said to belong to class $\mathcal{KL}$ if, for each fixed $s$, the map $\beta(r,s)$ belongs to class $\mathcal{K}$ with respect to $r$ and, for each fixed nonzero $r$, the map $\beta(r,s)$ is decreasing with respect to $s$ and $\beta(r,s)\rightarrow 0$ as \mbox{$s\rightarrow\infty$}. We identify a relation \mbox{$R\subseteq A\times B$} with the map \mbox{$R:A \rightarrow 2^{B}$} defined by $b\in R(a)$ iff \mbox{$(a,b)\in R$}. Given a relation \mbox{$R\subseteq A\times B$}, $R^{-1}$ denotes the inverse relation defined by \mbox{$R^{-1}=\{(b,a)\in B\times A:(a,b)\in R\}$}.

\subsection{Stochastic switched systems}\label{sss}
Let $(\Omega, \sigalg, \PP)$ be a probability space endowed with a filtration $\filtration = (\sigalg_t)_{t\geq 0}$ satisfying the usual conditions of completeness and right-continuity \cite[p.\ 48]{ref:KarShr-91}. Let $(W_t)_{t \ge 0}$ be a $\widehat{q}$-dimensional $\filtration$-adapted Brownian motion \cite{oksendal}. The class of stochastic switched systems considered in this paper is formalized as follows.  

\begin{definition}
\label{Def_control_sys}A stochastic switched system $\Sigma$ is a tuple $\Sigma=(\mathbb{R}^{n},\mathsf{P},\mathcal{P},F,G)$, where
\begin{itemize}
\item $\mathbb{R}^{n}$ is the state space;

\item $\mathsf{P}=\left\{1,\ldots,m\right\}$ is a finite set of modes;

\item $\mathcal{P}$ is a subset of the set of piecewise constant c\`adl\`ag (i.e. right-continuous and with left limits) functions from $\R_0^+$ to $\mathsf{P}$, and with a finite number of discontinuities on every bounded interval in $\R_0^+$ (no Zeno behaviour);  

\item $F=\left\{f_1,\ldots,f_m\right\}$ is such that for any $p\in\mathsf{P}$, $f_p:\R^n\rightarrow\R^n$ is globally Lipschitz continuous; 

\item $G=\left\{g_1,\ldots,g_m\right\}$ is such that for any $p\in\mathsf{P}$, $g_p:\R^n\rightarrow\R^{n\times{\widehat{q}}}$ is globally Lipschitz continuous with Lipschitz constant $Z_p\in\R_0^+$. 
\end{itemize}
\end{definition}

A continuous-time stochastic process \mbox{$\xi:\Omega \times \R_0^+ \rightarrow \mathbb{R}^{n}$} is said to be a \textit{solution process} of $\Sigma$ if there exists a switching signal $\upsilon\in\mathcal{P}$ satisfying 
\begin{small}
\begin{equation}
\label{eq00}
	\diff \xi= f_{\upsilon}(\xi)\diff t+g_{\upsilon}(\xi)\diff W_t,
\end{equation}
\end{small}$\PP$-almost surely ($\PP$-a.s.), at each time $t\in\R_0^+$ whenever $\upsilon$ is continuous. 
Let us emphasize that $\upsilon$ is a piecewise constant function defined over $\R_0^+$ and taking values in $\mathsf{P}$, 
which simply dictates in which mode the solution process $\xi$ is located, at any time $t \in \R_0^+$. 

For any given $p\in\mathsf{P}$, we denote by $\Sigma_p$ the subsystem of $\Sigma$ defined by the stochastic differential equation (SDE)
\begin{small}
\begin{equation}
\label{eq10}
	\diff \xi= f_{p}(\xi)\diff t+g_{p}(\xi)\diff W_t, 
\end{equation}
\end{small}where $f_{p}$ is known as the drift and
$g_{p}$ as the diffusion.
Given an initial condition which is a random variable, measurable in $\sigalg_0$, a solution process of $\Sigma_p$ exists and is uniquely determined owing to the assumptions on $f_p$ and on $g_p$ 
\cite[Theorem 5.2.1, p.\ 68]{oksendal}. 
 
We further write $\xi_{a \upsilon}(t)$ to denote the value of the solution process of $\Sigma$ at time $t\in\R_0^+$ under the switching signal $\upsilon$ from initial condition $\xi_{a \upsilon}(0) = a$ $\PP$-a.s., in which $a$ is a random variable that is measurable in $\sigalg_0$. 

Finally, note that a solution process of $\Sigma_p$ is also a solution process of $\Sigma$ corresponding to the constant switching signal $\upsilon(t)=p$, for all $t\in\R_0^+$. We also use $\xi_{ap}(t)$ to denote the value of the solution process of $\Sigma_p$ at time $t\in\R_0^+$ from the initial condition $\xi_{a p}(0) = a$ $\PP$-a.s..
\vspace{-0.2cm}

\section{Notions of Incremental Stability}\label{stability}
This section introduces some stability notions for stochastic switched systems, 
which generalize the notions of incremental global asymptotic stability ($\delta$-GAS) \cite{angeli} for non-probabilistic dynamical systems and of incremental global uniform asymptotic stability ($\delta$-GUAS) \cite{girard2} for non-probabilistic switched systems. 
The main results presented in this work rely on the stability assumptions discussed in this section. 

\begin{definition}
\label{dGAS}
The stochastic subsystem $\Sigma_p$ is incrementally globally asymptotically stable in the $q$th moment ($\delta$-GAS-M$_q$), where $q\geq1$, if there exists a $\mathcal{KL}$ function $\beta_p$ such that for any $t\in{\mathbb{R}_0^+}$ and any $\R^n$-valued random variables $a$ and $a'$ that are measurable in $\sigalg_0$, the following condition is satisfied:
\begin{small}
\begin{equation}
\EE \left[\left\Vert \xi_{ap}(t)-\xi_{a'p}(t)\right\Vert^q\right] \leq\beta_p\left( \EE\left[ \left\Vert a-a' \right\Vert^q\right], t \right). \label{delta_SGAS}
\end{equation}
\end{small}
\end{definition}

It can be easily checked that a $\delta$-GAS-M$_q$ stochastic subsystem $\Sigma_p$ is $\delta$-GAS \cite{angeli} in the absence of any noise. 
Further, note that when $f_p(0_n)=0_n$ and $g_p(0_n)=0_{n\times{\widehat{q}}}$ (drift and diffusion terms vanish at the origin), 
then $\delta$-GAS-M$_q$ implies global asymptotic stability in the $q$th moment (GAS-M$_q$) \cite{debasish}, 
which means that all the trajectories of $\Sigma_p$ converge in the $q$th moment to the (constant) trajectory $\xi_{0_np}(t)=0_n$ (namely, the equilibrium point), for all $t\in\R_0^+$. 
We extend the notion of $\delta$-GAS-M$_q$ to stochastic switched systems as follows.  

\begin{definition}
\label{dGUAS}
A stochastic switched system $\Sigma$ is incrementally globally uniformly asymptotically stable in the $q$th moment ($\delta$-GUAS-M$_q$), where $q\geq1$, if there exists a $\mathcal{KL}$ function $\beta$ such that for any $t\in{\mathbb{R}_0^+}$, any $\R^n$-valued random variables $a$ and $a'$ that are measurable in $\sigalg_0$, and any switching signal ${\upsilon}\in\mathcal{P}$, the following condition is satisfied:
\begin{small}
\begin{equation}
\EE \left[\left\Vert \xi_{a\upsilon}(t)-\xi_{a'\upsilon}(t)\right\Vert^q\right] \leq\beta\left( \EE\left[ \left\Vert a-a' \right\Vert^q\right], t \right). \label{delta_SGUAS}
\end{equation}
\end{small}
\end{definition}

Essentially, Definition \ref{dGUAS} extends Definition \ref{dGAS} uniformly over any possible switching signal $\upsilon\in\mathcal{P}$. 
As expected, this notion generalizes known ones in the literature:  
it can be easily seen that a $\delta$-GUAS-M$_q$ stochastic switched system $\Sigma$ is $\delta$-GUAS \cite{girard2} in the absence of any noise 
and that, whenever $f_p(0_n)=0_n$ and $g_p(0_n)=0_{n\times{\widehat{q}}}$ for all $p\in\mathsf{P}$, then $\delta$-GUAS-M$_q$ implies global uniform asymptotic stability in the $q$th moment (GUAS-M$_q$) \cite{debasish}. 

For non-probabilistic systems the $\delta$-GAS property can be characterized by $\delta$-GAS Lyapunov functions \cite{angeli}. 
Along these lines, we describe $\delta$-GAS-M$_q$ 
in terms of the existence of some {\em incremental Lyapunov functions}, defined as the following. 

\begin{definition}
\label{delta_SGAS_Lya}
Consider a stochastic subsystem $\Sigma_p$ and a continuous function $V_p:\mathbb{R}^n\times\mathbb{R}^n\rightarrow\mathbb{R}_0^+$ that is twice continuously differentiable on 
$\{\R^n\times\R^n\}\backslash\Delta$. 
Function $V_p$ is called a $\delta$-GAS-M$_q$ Lyapunov function for $\Sigma_p$, 
where $q\geq1$, 
if there exist $\mathcal{K}_{\infty}$ functions 
$\underline{\alpha}_p$, $\overline{\alpha}_p$ and a constant $\kappa_p\in\mathbb{R}^+$, such that
\begin{itemize}
\item[(i)] $\ul{\alpha}_p$ (resp. $\ol \alpha_p$) is a convex (resp. concave) function;

\item[(ii)] for any $x,x'\in\mathbb{R}^n$, $\underline{\alpha}_p\left(\left\Vert x-x'\right\Vert^q\right)\leq{V_p}(x,x')\leq\overline{\alpha}_p\left(\left\Vert x-x'\right\Vert^q\right)$;

\item[(iii)] for any $x,x'\in\mathbb{R}^n$, such that $x\neq x'$,
\begin{small}
\begin{align*}
	&\mathcal{L} V_p(x, x') 
	 \Let \left[\partial_xV_p~~\partial_{x'}V_p\right] \begin{bmatrix} f_p(x)\\f_p(x')\end{bmatrix}+\frac{1}{2} \text{Tr} \left(\begin{bmatrix} g_p(x) \\ g_p(x') \end{bmatrix}\left[g_p^T(x)~~g_p^T(x')\right] \begin{bmatrix}
\partial_{x,x} V_p & \partial_{x,x'} V_p\\ \partial_{x',x} V_p & \partial_{x',x'} V_p
\end{bmatrix}	\right)\leq-\kappa_pV_p(x,x'). 
\end{align*} 
\end{small}
\end{itemize}
\end{definition}
The operator $\mathcal{L}$ is the infinitesimal generator associated to the stochastic subsystem $\Sigma_p$, defined by the SDE in \eqref{eq10} \cite[Section 7.3]{oksendal}. 
The symbols $\partial_x$ and $\partial_{x,x'}$ denote first- and second-order partial derivatives with respect to $x$ and $x'$, respectively. 


The following theorem describes $\delta$-GAS-M$_q$ in terms of the existence of a $\delta$-GAS-M$_q$ Lyapunov function.

\begin{theorem}
\label{the_Lya}
A stochastic subsystem $\Sigma_p$ is $\delta$-GAS-M$_q$ if it admits a $\delta$-GAS-M$_q$ Lyapunov function. 
\end{theorem}

\begin{proof}
	The proof is a consequence of the application of Gronwall's inequality and of Ito's lemma \cite[p. 80 and 123]{oksendal}. Assume that there exists a $\delta$-GAS-M$_q$ Lyapunov function in the sense of Definition \ref{delta_SGAS_Lya}. For any $t\in\R_0^+$, and any $\R^n$-valued random variables $a$ and $a'$ that are measurable in $\sigalg_0$, we obtain
		\begin{small}
		\begin{align*}
			\EE \left[ V_p(\xi_{ap}(t),\xi_{a'p}(t)) \right]&=\EE \left[V_p(a,a') + \int_0^{t} \mathcal{L}V_p(\xi_{ap}(s),\xi_{a'p}(s))ds\right]\le\EE \left[ V_p(a,a') + \int_0^{t} \left(-\kappa_p V_p(\xi_{ap}(s),\xi_{a'p}(s))\right)ds\right] \\
			&\le-\kappa_p \int_0^{t} \EE \left[ V_p(\xi_{ap}(s),\xi_{a'p}(s))\right ]ds+\EE\left[V_p(a,a')\right]\nonumber,
		\end{align*}
		\end{small}which, by virtue of Gronwall's inequality, leads to 
		\begin{small}
\begin{align}
	\nonumber
		\EE &\left[ V_p(\xi_{ap}(t),\xi_{a'p}(t)) \right]\le \EE[V_p(a,a')] \mathsf{e}^{-\kappa_p t}.
\end{align}
\end{small}Hence, using property (ii) in Definition \ref{delta_SGAS_Lya}, we have
	\begin{small}
	\begin{align}\nonumber
		\ul{\alpha}_p\left(\EE \left[ \left\|\xi_{ap}(t)-\xi_{a'p}(t)\right\|^q\right]\right)&\le \EE \left[ \ul{\alpha}_p\left(\left\| \xi_{ap}(t)-\xi_{a'p}(t)\right\|^q\right) \right]\le\EE \left[ V_p\left(\xi_{ap}(t),\xi_{a'p}(t)\right)\right]\leq \EE\left[V_p(a,a')\right] \mathsf{e}^{-\kappa_p t}\\\notag&\leq \EE\left[\ol{\alpha}_p\left(\left\| a - a' \right\|^q\right)\right] \mathsf{e}^{-\kappa_p t}\le \ol{\alpha}_p\left( \EE\left[\left\| a - a' \right\|^q\right] \right) \mathsf{e}^{-\kappa_p t},
	\end{align}\end{small}where the first and last inequalities follow from property (i) and Jensen's inequality \cite[p. 310]{oksendal}. Since $\ul{\alpha}_p\in\mathcal{K}_\infty$, we obtain
	\begin{small}
\begin{align}\nonumber
\EE \left[ \left\|\xi_{ap}(t)-\xi_{a'p}(t)\right\|^q\right]\le\ul{\alpha}_p^{-1} \left(\ol{\alpha}_p\left( \EE\left[ \|a - a' \|^q\right] \right) \mathsf{e}^{-\kappa_p t}\right).
\end{align}
\end{small}Therefore, by introducing function
		$\beta_p\left(r,s\right) \Let \ul{\alpha}_p^{-1}\left(\ol{\alpha}_p\left(r\right)\mathsf{e}^{-\kappa_ps} \right)$,	
	condition (\ref{delta_SGAS}) is satisfied. Hence, the stochastic subsystem $\Sigma_p$ is $\delta$-GAS-M$_q$.	
\end{proof}

Let us now direct our attention from subsystems to the overall switched model. 
As qualitatively stated in the introduction, 
it is known that a non-probabilistic switched system, whose subsystems are all $\delta$-GAS, may exhibit some unstable behaviors under fast switching signals \cite{girard2} and, hence, may not be $\delta$-GUAS. 
The same phenomenon can happen for a stochastic switched system endowed by $\delta$-GAS-M$_q$ subsystems. 
The $\delta$-GUAS property of non-probabilistic switched systems can be established by using a common Lyapunov function, 
or alternatively via multiple Lyapunov functions that are mode-dependent \cite{girard2}.  
This leads to the following extensions for $\delta$-GUAS-M$_q$ property of stochastic switched systems. 

Assume that for any $p\in\mathsf{P}$, the stochastic subsystem $\Sigma_p$ admits a $\delta$-GAS-M$_q$ Lyapunov function $V_p$, satisfying conditions (i)-(iii) in Definition \ref{delta_SGAS_Lya} with $\mathcal{K}_\infty$ functions $\ul\alpha_p$, $\ol\alpha_p$, and a constant $\kappa_p\in\R^+$. 
Let us introduce the $\mathcal{K}_\infty$ functions $\ul\alpha$ and $\ol\alpha$ and the positive constant $\kappa$ for use in the rest of the paper as the following: $\ul\alpha=\min\left\{\ul\alpha_1,\ldots,\ul\alpha_m\right\}$, 
$\ol\alpha=\max\left\{\ol\alpha_1,\ldots,\ol\alpha_m\right\}$, 
and $\kappa=\min\left\{\kappa_1,\ldots,\kappa_m\right\}$. 
We first show a result based on the existence of a common Lyapunov function in which $\ul\alpha=\ul\alpha_1=\cdots=\ul\alpha_m$ and $\ol\alpha=\ol\alpha_1=\cdots=\ol\alpha_m$.  

\begin{theorem}\label{theorem2}
Consider a stochastic switched system $\Sigma$. 
If there exists a function $V$ that is a common $\delta$-GAS-M$_q$ Lyapunov function for all the subsystems $\left\{\Sigma_1,\ldots,\Sigma_m\right\}$, 
then $\Sigma$ is $\delta$-GUAS-M$_q$.
\end{theorem}

\begin{proof}
The proof is a consequence of the application of Gronwall's inequality and of Ito's lemma \cite[p. 80 and 123]{oksendal}. For any $\R^n$-valued random variables $a$ and $a'$ that are measurable in $\sigalg_0$, any switching signal $\upsilon\in\mathcal{P}$, and for all $t\in\R_0^+$ where $\upsilon$ is continuous, we have $\mathcal{L}V\left(\xi_{a\upsilon}(t),\xi_{a'\upsilon}(t)\right)\leq-\kappa V(\xi_{a\upsilon}(t),\xi_{a'\upsilon}(t))$. 
Using the continuity of $V$ and of the solution process $\xi$, for all $t\in\R_0^+$ one gets
\begin{small}
		\begin{align*}
			\EE \left[ V(\xi_{a\upsilon}(t),\xi_{a'\upsilon}(t)) \right]&\le\EE \left[ V(a,a') + \int_0^{t} \left(-\kappa V(\xi_{a\upsilon}(s),\xi_{a'\upsilon}(s))\right)ds\right] \le-\kappa\int_0^{t} \EE \left[ V(\xi_{a\upsilon}(s),\xi_{a'\upsilon}(s))\right ]ds+\EE\left[V(a,a')\right]\nonumber,
		\end{align*}
		\end{small}which, by virtue of Gronwall's inequality, leads to 
		\begin{small}
\begin{align}
	\nonumber
		\EE &\left[ V(\xi_{a\upsilon}(t),\xi_{a'\upsilon}(t)) \right]\le \EE[V(a,a')] \mathsf{e}^{-\kappa t}.
\end{align}
\end{small}
Since the $\mathcal{K}_\infty$ functions $\ul\alpha$ and $\ol\alpha$ are convex and concave, respectively, using Jensen's inequality we have
\begin{small}
	\begin{align}\nonumber
		\ul\alpha\left(\EE \left[ \left\Vert\xi_{a\upsilon}(t)-\xi_{a'\upsilon}(t)\right\Vert^q\right]\right)&\le \EE \left[ \ul\alpha\left(\left\Vert \xi_{a\upsilon}(t)-\xi_{a'\upsilon}(t)\right\Vert^q\right) \right]\le\EE \left[ V\left(\xi_{a\upsilon}(t),\xi_{a'\upsilon}(t)\right)\right]\leq \EE\left[V(a,a')\right] \mathsf{e}^{-\kappa t}\\\notag&\leq \EE\left[\ol\alpha\left(\left\Vert a - a' \right\Vert^q\right)\right] \mathsf{e}^{-\kappa t}\le \ol\alpha\left( \EE\left[\left\Vert a - a' \right\Vert^q\right] \right) \mathsf{e}^{-\kappa t}.
	\end{align}
	\end{small}
Since $\ul\alpha\in\mathcal{K}_\infty$, we obtain \begin{small}$$\EE \left[ \left\|\xi_{a\upsilon}(t)-\xi_{a'\upsilon}(t)\right\|^q\right]\le\ul{\alpha}^{-1} \left(\ol{\alpha}\left( \EE\left[ \|a - a' \|^q\right] \right) \mathsf{e}^{-\kappa t}\right),$$\end{small}for all $t\in\R_0^+$. Then condition (\ref{delta_SGUAS}) holds with the function $\beta(r,s)\Let\ul\alpha^{-1}\left(\ol\alpha({r})\mathsf{e}^{-\kappa s}\right)$. 
\end{proof}

When a common $\delta$-GAS-M$_q$ Lyapunov function $V$ fails to exist, 
the $\delta$-GUAS-M$_q$ property of $\Sigma$ can still be established by resorting to multiple $\delta$-GAS-M$_q$ Lyapunov functions (one per mode) over a restricted set of switching signals. 
More precisely, let $\mathcal{P}_{\tau_d}$ be a subset of the set of switching signals $\upsilon$ with \emph{dwell time} 
$\tau_d\in\R_0^+$, 
where $\upsilon$ is said to have dwell time $\tau_d$ if the switching times $t_1,t_2,\ldots$ (occurring at the discontinuity points of $\upsilon$) satisfy $t_1>\tau_d$ and $t_i-t_{i-1}\geq\tau_d$, for all $i\geq{2}$. 
We now show a stability result based on the existence of multiple Lyapunov functions.

\begin{theorem}\label{multiple_lyapunov}
Let $\tau_d\in\R_0^+$, and consider a stochastic switched system $\Sigma_{\tau_d}=(\mathbb{R}^{n},\mathsf{P},\mathcal{P}_{\tau_d},F,G)$. Assume that for any $p\in\mathsf{P}$, there exists a $\delta$-GAS-M$_q$ Lyapunov function $V_p$ for subsystem $\Sigma_{\tau_d,p}$ and that in addition there exits a constant $\mu\geq1$ such that
\begin{small}
\begin{equation}\label{eq0}
\forall{x,x'}\in\R^n,~~\forall{p,p'\in\mathsf{P}},~~V_p(x,x')\leq\mu V_{p'}(x,x').
\end{equation}
\end{small}
If $\tau_d>\log{\mu}/\kappa$, then $\Sigma_{\tau_d}$ is $\delta$-GUAS-M$_q$.
\end{theorem}

\begin{proof}
The proof is inspired by that of Theorem 2.8 in \cite{girard2} for the non-probabilistic case. We show the result for the case that switching signals have infinite number of discontinuities (switching times). A proof for the case of finite discontinuities can be written in a similar way. Let $a$ and $a'$ be any $\R^n$-valued random variables that are measurable in $\sigalg_0$, $\upsilon\in\mathcal{P}_{\tau_d}$, $t_0=0$, and let $p_{i+1}\in\mathsf{P}$ denote the value of the switching signal on the open interval $(t_i,t_{i+1})$, for $i\in\N_0$. Using (iii) in Definition \ref{delta_SGAS_Lya} for all $i\in\N_0$ and $t\in(t_i,t_{i+1})$, one gets\begin{small}$$\mathcal{L} V_{p_{i+1}}\left(\xi_{a\upsilon}(t),\xi_{a'\upsilon}(t)\right)\leq-\kappa V_{p_{i+1}}\left(\xi_{a\upsilon}(t),\xi_{a'\upsilon}(t)\right).$$\end{small} Similar to the proof of Theorem \ref{theorem2}, for all $i\in\N_0$ and $t\in[t_i,t_{i+1}]$, we have 
\begin{small}
\begin{align}\label{eq1}
\EE&\left[V_{p_{i+1}}(\xi_{a\upsilon}(t),\xi_{a'\upsilon}(t))\right]\leq\EE\left[V_{p_{i+1}}(\xi_{a\upsilon}(t_i),\xi_{a'\upsilon}(t_i))\right]\mathsf{e}^{-\kappa(t-t_i)}. 
\end{align}
\end{small}
Particularly, for $t=t_{i+1}$ and from (\ref{eq0}), it can be checked that for all $i\in\N_0$:
\begin{small}
\begin{align*}
\EE&\left[V_{p_{i+2}}(\xi_{a\upsilon}(t_{i+1}),\xi_{a'\upsilon}(t_{i+1}))\right]\leq\mu\mathsf{e}^{-\kappa(t_{i+1}-t_i)}\EE\left[V_{p_{i+1}}(\xi_{a\upsilon}(t_i),\xi_{a'\upsilon}(t_i))\right]. 
\end{align*}
\end{small}
Using this inequality, we prove by induction that for all $i\in\N_0$
\begin{small}
\begin{equation}\label{eq2}
\EE\left[V_{p_{i+1}}(\xi_{a\upsilon}(t_i),\xi_{a'\upsilon}(t_i))\right]\leq\mu^i\mathsf{e}^{-\kappa{t_i}}\EE\left[V_{p_1}(a,a')\right].
\end{equation} 
\end{small}
From (\ref{eq1}) and (\ref{eq2}), for all $i\in\N_0$ and $t\in[t_i,t_{i+1}]$, one obtains\begin{small}$$\EE\left[V_{p_{i+1}}(\xi_{a\upsilon}(t),\xi_{a'\upsilon}(t))\right]\leq\mu^i\mathsf{e}^{-\kappa{t}}\EE\left[V_{p_1}(a,a')\right].$$\end{small} Since the switching signal $\upsilon$ has dwell time $\tau_d$, then $t_i\geq i\tau_d$ and hence for all $t\in[t_i,t_{i+1}]$, $t\geq i\tau_d$. Since $\mu\geq1$, then for all $i\in\N_0$ and $t\in[t_i,t_{i+1}]$, one has $\mu^i=\mathsf{e}^{i\log{\mu}}\leq\mathsf{e}^{\left(\log{\mu}/\tau_d\right)t}.$ Therefore, for all $i\in\N_0$ and $t\in[t_i,t_{i+1}]$, we get\begin{small}$$\EE\left[V_{p_{i+1}}(\xi_{a\upsilon}(t),\xi_{a'\upsilon}(t))\right]\leq\mathsf{e}^{\left(\left(\log{\mu}/\tau_d\right)-\kappa\right)t}\EE\left[V_{p_1}(a,a')\right].$$\end{small} Using functions $\ul\alpha,\ol\alpha$ and Jensen's inequality, and for all $t\in\R_0^+$, where $t\in[t_i,t_{i+1}]$ for some $i\in\N_0$, we have
\begin{small}
\begin{align}\nonumber
\ul\alpha\left(\EE\left[\left\Vert\xi_{a\upsilon}(t)-\xi_{a'\upsilon}(t)\right\Vert^q\right]\right)&\leq\ul\alpha_{p_{i+1}}\left(\EE\left[\left\Vert\xi_{a\upsilon}(t)-\xi_{a'\upsilon}(t)\right\Vert^q\right]\right)\leq\EE\left[\ul\alpha_{p_{i+1}}\left(\left\Vert\xi_{a\upsilon}(t)-\xi_{a'\upsilon}(t)\right\Vert^q\right)\right]\leq\EE\left[V_{p_{i+1}}(\xi_{a\upsilon}(t),\xi_{a'\upsilon}(t))\right]\\\notag&\leq\mathsf{e}^{\left(\left(\log{\mu}/\tau_d\right)-\kappa\right)t}\EE\left[V_{p_1}(a,a')\right]\leq\mathsf{e}^{\left(\left(\log{\mu}/\tau_d\right)-\kappa\right)t}\EE\left[\ol\alpha_{p_1}\left(\Vert a-a'\Vert^q\right)\right]\\\notag&\leq\mathsf{e}^{\left(\left(\log{\mu}/\tau_d\right)-\kappa\right)t}\ol\alpha_{p_1}\left(\EE\left[\Vert a-a'\Vert^q\right]\right)\leq\mathsf{e}^{\left(\left(\log{\mu}/\tau_d\right)-\kappa\right)t}\ol\alpha\left(\EE\left[\left\Vert a-a'\right\Vert^q\right]\right).
\end{align}
\end{small}
Therefore, for all $t\in\R_0^+$
\begin{small}
\begin{align*}
\EE&\left[\left\Vert\xi_{a\upsilon}(t)-\xi_{a'\upsilon}(t)\right\Vert^q\right]\leq\ul\alpha^{-1}\left(\mathsf{e}^{\left(\left(\log{\mu}/\tau_d\right)-\kappa\right)t}\ol\alpha\left(\EE\left[\Vert a-a'\Vert^q\right]\right)\right). 
\end{align*}
\end{small}
Then condition (\ref{delta_SGUAS}) holds with the function $\beta(r,s)\Let\ul\alpha^{-1}\left(\ol\alpha({r})\mathsf{e}^{\left(\left(\log{\mu}/\tau_d\right)-\kappa\right)s}\right)$ which is a $\mathcal{KL}$ function since by assumption $\log\mu/\tau_d-\kappa<0$. The same inequality holds for switching signals with a finite number of discontinuities, 
hence the stochastic switched system $\Sigma_{\tau_d}$ is $\delta$-GUAS-M$_q$. 
\end{proof}

In order to show some of the main results of the paper in Section \ref{existence}, 
we need the following technical result, 
which provides an upper bound on the distance (in the $q$th moment) between the solution processes of $\Sigma_p$ (resp. $\Sigma_{\tau_d,p}$) and the corresponding non-probabilistic subsystem $\ol\Sigma_p$ (resp. $\ol\Sigma_{\tau_d,p}$), obtained by disregarding the diffusion term $g_p$. 
From now on, we use the notation $\ol\xi_{xp}$ to denote the trajectory of $\ol\Sigma_p$ (resp. $\ol\Sigma_{\tau_d,p}$) starting from the initial condition $x$ and satisfying the ordinary differential equation (ODE) $\dot{\ol\xi}_{xp}=f_p\left(\ol\xi_{xp}\right)$. 

\begin{lemma}
\label{lem:moment est}
	Consider a stochastic subsystem $\Sigma_p$ (resp. $\Sigma_{\tau_d,p}$) such that $f_p(0_n)=0_n$ and $g_p(0_n) = 0_{n\times{\widehat{q}}}$.	
	Suppose $q\geq2$ and there exists a $\delta$-GAS-M$_q$ Lyapunov function $V_p$ for $\Sigma_p$ (resp. $\Sigma_{\tau_d,p}$) such that its Hessian is a positive semidefinite matrix in $\R^{2n\times2n}$ and $\partial_{x,x}V_p(x,x')\leq P_p$, $\forall x,x'\in\R^n$ and some positive semidefinite matrix $P_p\in\R^{n\times{n}}$. Then for any $x\in\R^n$, we have \begin{small}$\EE \left[\left\Vert\traj{\xi}{x}{p}(t)-\traj{\ol \xi}{x}{p}(t)\right\Vert^q\right] \le h_x^p(t)$\end{small}, 
	where 
	\begin{small}
	\begin{align}\nonumber
	h_x^p(t)=&\ul\alpha_p^{-1}\left(\frac{1}{2}\left\Vert{\sqrt{P_p}}\right\Vert^2\min\{n,\widehat{q}\}Z_p^2\mathsf{e}^{-\kappa_p t}\int_0^t\left(\beta_p\left(\left\Vert{x}\right\Vert^q,s\right)\right)^{\frac{2}{q}}ds\right),
	\end{align}
	\end{small}$Z_p$ is the Lipschitz constant, introduced in Definition \ref{Def_control_sys}, and $\beta_p$ is the $\mathcal{KL}$ function\footnote{Using a $\delta$-GAS-M$_q$ Lyapunov function $V_p$, one can always choose $\beta_p(r,s)=\ul{\alpha}_p^{-1}\left(\ol{\alpha}_p\left(r\right)\mathsf{e}^{-\kappa_ps}\right)$, as showed in Theorem \ref{the_Lya}.} appearing in \eqref{dGAS}.
\end{lemma}

It can be readily seen that the nonnegative valued function $h_x^p$ tends to zero as $t \ra 0$, $t\ra+\infty$, or as $Z_p\ra0$ and is identically zero if the diffusion term is identically zero (i.e. $Z_p=0$) which is the case for $\ol\Sigma_p$ (resp. $\ol\Sigma_{\tau_d,p}$).

\begin{proof}
The proof is similar to the proof of Lemma 3.7 in \cite{majid8}, where one needs to eliminate all the terms $\gamma(\cdot)$. 
\end{proof}

The interested readers are referred to the results in \cite{majid8}, providing a result in line with that of Lemma \ref{lem:moment est} for an (affine) stochastic subsystem $\Sigma_p$ (resp. $\Sigma_{\tau_d,p}$) admitting a specific type of $\delta$-GAS-M$_q$ Lyapunov functions. 
For later use, we introduce function $h_x(t)=\max\left\{h_x^1(t),\ldots,h_x^m(t)\right\}$ for all $t\in\R_0^+$.
\vspace{-0.2cm}

\section{Systems and Approximate Equivalence Notions}\label{symbolic}
We employ the notion of \emph{system}, introduced in \cite{paulo}, to provide (in Sec. \ref{existence}) an alternative description of stochastic switched systems that can be later directly related to their symbolic models. 
\begin{definition}\label{system}
A system $S$ is a tuple 
$S=(X,X_0,U,\longrightarrow,Y,H),$
where
$X$ is a set of states (possibly infinite), 
$X_0\subseteq X$ is a set of initial states (possibly infinite), 
$U$ is a set of inputs (possibly infinite),
$\longrightarrow\subseteq X\times U\times X$ is a transition relation,
$Y$ is a set of outputs, and
$H:X\rightarrow Y$ is an output map.
\end{definition}

We write $x\rTo^{u}x'$ if \mbox{$(x,u,x')\in\longrightarrow$}. 
If $x\rTo^{u}x'$, we call state  $x'$ a \mbox{$u$-successor}, or simply a successor, of state $x$.
For technical reasons, we assume that for each $x\in X$, there is some $u$-successor of $x$, for some $u\in U$ -- let us remark that this is always the case for the considered systems later in this paper. 

A system $S$ is said to be
\begin{itemize}
\item \textit{metric}, if the output set $Y$ is equipped with a metric
$\mathbf{d}:Y\times Y\rightarrow\mathbb{R}_{0}^{+}$;

\item \textit{finite} (or \textit{symbolic}), if $X$ and $U$ are finite sets;

\item \textit{deterministic}, if for any state $x\in{X}$ and any input $u\in{U}$, there exists at most one \mbox{$u$-successor}.
\end{itemize}

For a system $S=(X,X_0,U,\longrightarrow,Y,H)$, given any initial state $x_0\in{X_0}$, a finite state run generated from $x_0$ is a finite sequence of transitions: 
\begin{small}
\begin{align}
\label{run}
x_0\rTo^{u_0}x_1\rTo^{u_1}x_2\rTo^{u_2}\cdots\rTo^{u_{n-2}}x_{n-1}\rTo^{u_{n-1}}x_n,
\end{align}
\end{small}such that $x_i\rTo^{u_i}x_{i+1}$ for all $0\leq i<n$. A finite state run can be trivially extended to an infinite state run \cite{paulo}. A finite output run is a sequence $\left\{y_0,y_1,\ldots,y_n\right\}$ such that there exists a finite state run of the form \eqref{run} with $y_i=H(x_i)$, for $i=1,\ldots,n$. A finite output run can also be directly extended to an infinite output run as well \cite{paulo}.

Now, we recall the notion of approximate (bi)simulation relation, introduced in \cite{girard}, 
which is useful when analyzing or synthesizing controllers for deterministic systems.  
 
\begin{definition}\label{APSR}
Consider metric systems $S_{a}=(X_{a},X_{a0},U_{a},\rTo_{a},Y_a,H_{a})$ and
$S_{b}=(X_{b},X_{b0},U_{b},\rTo_{b},Y_b,H_{b})$ with the
same output sets $Y_a=Y_b$ and metric $\mathbf{d}$.
For $\varepsilon\in\mathbb{R}_0^{+}$,
a relation
\mbox{$R\subseteq X_{a}\times X_{b}$} is said to be an $\varepsilon$-approximate simulation relation from $S_{a}$ to $S_{b}$
if for all $(x_a,x_b)\in R$ the following two conditions are satisfied:
\begin{itemize}
\item[(i)] \mbox{$\mathbf{d}(H_{a}(x_{a}),H_{b}(x_{b}))\leq\varepsilon$};
\item[(ii)] \mbox{$x_{a}\rTo_{a}^{u_a}x'_{a}$ in $S_a$} implies the existence of \mbox{$x_{b}\rTo_{b}^{u_b}x'_{b}$} in $S_b$ satisfying $(x'_{a},x'_{b})\in R$.
\end{itemize}
A relation $R\subseteq X_a\times X_b$ is said to be an $\varepsilon$-approximate bisimulation relation between $S_a$ and $S_b$
if $R$ is an $\varepsilon$-approximate simulation relation from $S_a$ to $S_b$ and
$R^{-1}$ is an $\varepsilon$-approximate simulation relation from $S_b$ to $S_a$.

System $S_{a}$ is $\varepsilon$-approximately simulated by $S_{b}$, or $S_b$ $\varepsilon$-approximately simulates $S_a$,
denoted by \mbox{$S_{a}\preceq_{\mathcal{S}}^{\varepsilon}S_{b}$}, if there exists an $\varepsilon$-approximate simulation relation $R$ from $S_a$ to $S_b$ such that:
\begin{itemize}
\item $\forall x_{a0}\in{X_{a0}}$, $\exists x_{b0}\in{X_{b0}}$ with $(x_{a0},x_{b0})\in{R}$.
\end{itemize}
System $S_{a}$ is $\varepsilon$-approximately bisimilar to $S_{b}$, denoted by \mbox{$S_{a}\cong_{\mathcal{S}}^{\varepsilon}S_{b}$}, if there exists an $\varepsilon$-approximate bisimulation relation $R$ between $S_a$ and $S_b$ such that:
\begin{itemize}
\item $\forall x_{a0}\in{X_{a0}}$, $\exists x_{b0}\in{X_{b0}}$ with $(x_{a0},x_{b0})\in{R}$;
\item $\forall x_{b0}\in{X_{b0}}$, $\exists x_{a0}\in{X_{a0}}$ with $(x_{a0},x_{b0})\in{R}$.
\end{itemize}
\end{definition}

\section{Symbolic Models for Stochastic Switched Systems}\label{existence}
The main results of this work are presented in this section. 
We show that for any stochastic switched system $\Sigma$ (resp. $\Sigma_{\tau_d}$), 
admitting a common (resp. multiple) $\delta$-GAS-M$_q$ Lyapunov function(s), 
and for any precision level $\varepsilon\in\R^+$, one can construct a finite system that is $\varepsilon$-approximately bisimilar to $\Sigma$ (resp. $\Sigma_{\tau_d}$). 

\subsection{Describing stochastic switched systems as metric systems}
In order to show the main results of the paper, 
we use systems as an abstract representation of stochastic switched systems, capturing all the information they contain at the sampling times.   
More precisely, given a stochastic switched system $\Sigma$ and a sampling time $\tau$, 
we define an associated metric system $S_\tau(\Sigma)=(X_\tau,X_{\tau0},U_\tau,\rTo_\tau,Y_\tau,H_\tau),$ 
where:
\begin{itemize}
\item $X_\tau$ is the set of all $\R^n$-valued random variables defined on 
the probability space $(\Omega,\sigalg,\PP)$;

\item $X_{\tau0}\subseteq\mathcal{X}_0$, where $\mathcal{X}_0$ is the set of all $\R^n$-valued random variables that are measurable over $\sigalg_0$;

\item $U_\tau=\mathsf{P}$;

\item $x_\tau\rTo^{p}_\tau x'_\tau$ if $x_\tau$ and $x'_\tau$ are measurable, 
respectively, in $\sigalg_{k\tau}$ and $\sigalg_{(k+1)\tau}$ for some $k \in\N_0$, and 
there exists a solution process $\xi:\Omega\times\R_0^+\rightarrow\R^n$ of $\Sigma$ satisfying $\xi(k\tau) = x_\tau$ and $\xi_{x_\tau p}(\tau) = x'_\tau$ $\PP$-a.s.;

\item $Y_\tau=X_\tau$;

\item $H_\tau(x_\tau)=x_{\tau}$ for any $x_\tau\in X_\tau$.
\end{itemize}
We assume that the output set $Y_\tau$ is equipped with the metric $\mathbf{d}(y,y')=\left(\EE\left[\left\Vert y-y'\right\Vert^q\right]\right)^{\frac{1}{q}}$, for any $y,y'\in{Y_\tau}$ and some $q\geq1$. Let us remark that the set of states of $S_\tau(\Sigma)$ is uncountable and that ${S_{\tau}}(\Sigma)$ is a deterministic system in the sense of Definition \ref{system}, since (cf. Subsection \ref{sss})  
its solution process is uniquely determined given an initial condition. 

Note that a finite state run
$x_{0}\rTo^{p_0}_{\tau}x_{1}\rTo^{p_1}_{\tau} \,\cdots\, \rTo^{p_{N-1}}_{\tau} x_{N}$ of $S_{\tau}(\Sigma)$, where $p_{i-1}\in\mathsf{P}$ and $x_i=\xi_{x_{i-1}p_{i-1}}(\tau)$ $\PP$-a.s. for $i=1,\ldots,N$, captures the trajectory of the stochastic switched system $\Sigma$ at times $t=0,\tau,\ldots,N\tau$. This trajectory is  
started from the initial condition $x_{0}\in X_{\tau0}$ and resulting from a switching signal $\upsilon$ obtained by the concatenation of the modes  
$p_{i-1}$ \big(i.e. $\upsilon(t)=p_{i-1}$ for any $t\in [(i-1)\tau,i\,\tau[$\big), for $i=1,\ldots,N$.

Now we represent a stochastic switched system $\Sigma_{\tau_d}$ with a metric system where, without loss of generality, 
we assume that $\tau_d$ is an integer multiple of $\tau$, 
i.e. $\exists\widehat{N}\in\N$ such that $\tau_d=\widehat{N}\tau$. Given a stochastic switched system $\Sigma_{\tau_d}$ and a sampling time $\tau\in\R^+$, 
we define the metric system $S_\tau\left(\Sigma_{\tau_d}\right)=(X_\tau,X_{\tau0},U_\tau,\rTo_\tau,Y_\tau,H_\tau),$ 
where:
\begin{itemize}
\item $X_\tau=\mathcal{X}\times\mathsf{P}\times\{0,\ldots,\widehat{N}-1\}$, where $\mathcal{X}$ is the set of all $\R^n$-valued random variables defined on the probability space $(\Omega,\sigalg,\PP)$;

\item $X_{\tau0}\subseteq\mathcal{X}_0\times\mathsf{P}\times\{0,\ldots,\widehat{N}-1\}$, where $\mathcal{X}_0$ is the set of all $\R^n$-valued random variables that are measurable over $\sigalg_0$;

\item $U_\tau=\mathsf{P}$;

\item $\left(x_\tau,p,i\right)\rTo^{p}_\tau\left({x'_\tau},p',i'\right)$ if $x_\tau$ and $x'_\tau$ are measurable, 
respectively, in $\sigalg_{k\tau}$ and $\sigalg_{(k+1)\tau}$ for some $k \in \N_0$, and 
there exists a solution process $\xi:\Omega\times\R_0^+\rightarrow\R^n$ of $\Sigma_{\tau_d}$ satisfying $\xi(k\tau) = x_\tau$ and $\xi_{x_\tau p}(\tau) = x'_\tau$ $\PP$-a.s. and one of the following conditions hold: 

\begin{itemize}
\item $i<\widehat{N}-1$, $p'=p$, and $i'=i+1$: switching is not allowed because the time elapsed since the latest switch is strictly smaller than the dwell time;

\item $i=\widehat{N}-1$, $p'=p$, and $i'=\widehat{N}-1$: switching is allowed but no switch occurs;

\item $i=\widehat{N}-1$, $p'\neq p$, and $i'=0$: switching is allowed and a switch occurs. 
\end{itemize}

\item $Y_\tau=\mathcal{X}$;

\item $H_\tau\left(x_\tau,p,i\right)=x_\tau$ for any $(x_\tau,p,i)\in X_\tau$.
\end{itemize}
We assume that the output set $Y_\tau$ is equipped with the metric $\mathbf{d}(y,y')=(\EE[\left\Vert y-y'\right\Vert^q])^{\frac{1}{q}}$, $\forall y,y'\in{Y_\tau}$ and for some $q\geq1$. One can readily verify that the (in)finite output runs of $S_{\tau}\left(\Sigma_{\tau_d}\right)$ are the (in)finite output runs of $S_\tau(\Sigma)$ corresponding to switching signals with dwell time $\tau_d=\widehat{N}\tau$.

In order to show the main results of this work, 
we assume that for any $\delta$-GAS-M$_q$ Lyapunov functions $V_p$, there exists a $\mathcal{K}_\infty$ and concave function $\widehat\gamma_p$ such that
\begin{equation}\label{supplement}
\vert V_p(x,y)-V_p(x,z)\vert\leq\widehat\gamma_p\left(\Vert y-z\Vert\right),
\end{equation}
for any $x,y,z\in\R^n$. 
This assumption is not restrictive at all, provided the function $V_p$ is limited to a compact subset of $\R^n\times\R^n$. For all $x,y,z\in\mathsf{D}$, where $\mathsf{D}$ is a compact subset of $\R^n$, by applying the mean value theorem 
to the function $y\rightarrow V_p(x,y)$, one gets 
\begin{small}
\begin{align}\notag
&\left\vert V_p(x,y)-V_p(x,z)\right\vert\leq\widehat\gamma_p\left(\Vert y-z\Vert\right),\text{s.t.}~\widehat\gamma_p({r})=\left(\max_{(x,y)\in\mathsf{D}\backslash\Delta}\left\Vert\frac{\partial{V_p}(x,y)}{\partial{y}}\right\Vert\right)r.
\end{align}
\end{small}
For later use, let us define the $\mathcal{K}_\infty$ function $\widehat\gamma$ such that $\widehat\gamma=\max\left\{\widehat\gamma_1,\ldots,\widehat\gamma_m\right\}$. (Note that, for the case of a common Lyapunov function, we have: $\widehat\gamma=\widehat\gamma_1=\cdots=\widehat\gamma_m$.) 
We proceed presenting the main results of this work.  

\subsection{First approach}
This subsection contains the first main results of the paper which are based on the state space discretization. For later use in this subsection, let us define the function $h_X(t)=\max_{x\in X}h_x(t)$, for a set $X\subseteq \R^n$.

\subsubsection{Common Lyapunov function}
We show the first result on finite abstractions based on the existence of a common $\delta$-GAS-M$_q$ Lyapunov function for subsystems $\Sigma_1,\ldots,\Sigma_m$. 
Consider a stochastic switched system $\Sigma$ 
and a pair $\mathsf{q}=(\tau,\eta)$ of quantization parameters, 
where $\tau$ is the sampling time and $\eta$ is the state space quantization. 
Given $\Sigma$ and $\mathsf{q}$, consider the following system:
\begin{small}
\begin{equation}\label{T2}
S_{\mathsf{q}}(\Sigma)=(X_{\mathsf{q}},X_{\params0},U_{\mathsf{q}},\rTo_{\mathsf{q}},Y_{\mathsf{q}},H_{\mathsf{q}}),
\end{equation}
\end{small}where $X_{\mathsf{q}}=[\R^n]_\eta$, $X_{\params0}=[\R^n]_\eta$, $U_{\mathsf{q}}=\mathsf{P}$, $Y_{\mathsf{q}}=Y_\tau$, and
\begin{itemize}
\item $x_{\mathsf{q}}\rTo_{\mathsf{q}}^{p}x'_{\mathsf{q}}$ if there exists a $x'_\params\in X_\params$ such that $\left\Vert\ol{\xi}_{x_{\mathsf{q}}p}(\tau)-x'_{\mathsf{q}}\right\Vert\leq\eta$;


\item $H_{\mathsf{q}}(x_\params)=x_\params$ for any $x_\params\in X_{\mathsf{q}}$.
\end{itemize}

In order to relate models, 
the output set $Y_\params$ is taken to be that of the system $S_\tau(\Sigma)$. 
Therefore, $H_\params$, with a slight abuse of notation, is a mapping from a grid point to a random variable with a Dirac probability distribution centered at the grid point. 


We now present the first main result of the paper. In order to show the next result, we assume that $f_p(0_n)=0_n$ only if $\Sigma_p$ is not affine and that $g_p(0_n)=0_{n\times\widehat{q}}$ for any $p\in\mathsf{P}$.

\begin{theorem}\label{main_theorem1}
Let $\Sigma$ be a stochastic switched system admitting a common $\delta$-GAS-M$_q$ Lyapunov function $V$, 
of the form discussed in Lemma \ref{lem:moment est}, for subsystems $\Sigma_1,\ldots,\Sigma_m$. For $X_{\tau0}=\R^n$, any $\varepsilon\in\R^+$, and any double $\mathsf{q}=(\tau,\eta)$ of quantization parameters satisfying
\begin{small}
\begin{align}\label{bisim_cond1}
\overline\alpha\left(\eta^q\right)&\leq\underline\alpha\left(\varepsilon^q\right),\\\label{bisim_cond}
\mathsf{e}^{-\kappa\tau}\underline\alpha\left(\varepsilon^q\right)+\widehat\gamma\left(\left(h_{[X_{\tau0}]_\eta}(\tau)\right)^{\frac{1}{q}}+\eta\right)&\leq\underline\alpha\left(\varepsilon^q\right),
\end{align}
\end{small}
we have that \mbox{$S_{\mathsf{q}}(\Sigma)\cong_{\mathcal{S}}^{\varepsilon}S_{\tau}(\Sigma)$}.
\end{theorem} 

It can be readily seen that when we are interested in the dynamics of $\Sigma$ on a compact $\mathsf{D}\subset\R^n$ of the form of a finite union of boxes, implying that $X_{\tau0}=\mathsf{D}$, and for a given precision $\varepsilon$,
there always exists a sufficiently large value of $\tau$ and a small value of $\eta$ such that $\eta\leq\boxspan(\mathsf{D})$ and the conditions in (\ref{bisim_cond1}) and (\ref{bisim_cond}) are satisfied. For a given fixed sampling time $\tau$, the precision $\varepsilon$ is lower bounded by:
\begin{small}
\begin{equation}\label{lower_bound}
\varepsilon>\left(\ul\alpha^{-1}\left(\frac{\widehat\gamma\left(\left(h_{[X_{\tau0}]_\eta}(\tau)\right)^{\frac{1}{q}}\right)}{1-\mathsf{e}^{-\kappa\tau}}\right)\right)^{\frac{1}{q}}.
\end{equation}
\end{small} 
One can easily verify that the lower bound on $\varepsilon$ in (\ref{lower_bound}) goes to zero as $\tau$ goes to infinity or as $Z_p \ra 0$, for any $p\in\mathsf{P}$, where $Z_p$ is the Lipschitz constant introduced in Definition \ref{Def_control_sys}. 

Note that $S_\params(\Sigma)$ has a countable number of states and it is finite if one is interested in the dynamics of $\Sigma$ on a compact $\mathsf{D}\subset\R^n$ which is always the case in practice.

\begin{proof}
We start by proving \mbox{$S_{\tau}(\Sigma)\preceq^{\varepsilon}_\mathcal{S}S_{\params}(\Sigma)$}.
Consider the relation $R\subseteq X_{\tau}\times X_{\params}$ defined by 
{$\left(x_{\tau},x_{\params}\right)\in R$}
if and only if 
{$\mathbb{E}\left[V\left(H_{\tau}(x_{\tau}),H_{\params}(x_{\params})\right)\right]=\mathbb{E}\left[V\left(x_{\tau},x_{\params}\right)\right]\leq\underline\alpha\left(\varepsilon^q\right)$}. Consider any \mbox{$\left(x_{\tau},x_{\params}\right)\in R$}. Condition (i) in Definition \ref{APSR} is satisfied because
\begin{small}
\begin{equation} 
\label{convexity1}
\left(\mathbb{E}\left[\Vert x_{\tau}-x_{\params}\Vert^q\right]\right)^{\frac{1}{q}}\leq\left(\underline\alpha^{-1}\left(\mathbb{E}\left[V(x_{\tau},x_{\params})\right]\right)\right)^{\frac{1}{q}}\leq\varepsilon.
\end{equation}
\end{small}
We used the convexity assumption of $\underline\alpha$ and the Jensen inequality \cite{oksendal} to show the inequalities in (\ref{convexity1}). Let us now show that condition (ii) in Definition
\ref{APSR} holds. 
Consider the transition \mbox{$x_{\tau}\rTo^{p}_{\tau} x'_{\tau}=\xi_{x_{\tau}p}(\tau)$} $\PP$-a.s. in $S_{\tau}(\Sigma)$. Since $V$ is a common Lyapunov function for $\Sigma$, we have
\begin{small}
\begin{align}\label{b02}
\mathbb{E}\left[V(x'_{\tau},\xi_{x_{\params}p}(\tau))\right] &\leq \EE\left[V(x_\tau,x_q)\right] \mathsf{e}^{-\kappa\tau}\leq \underline\alpha\left(\varepsilon^q\right) \mathsf{e}^{-\kappa\tau}.
\end{align}
\end{small}
Since \mbox{$\R^n\subseteq\bigcup_{p\in[\mathbb{R}^n]_{\eta}}\mathcal{B}_{\eta}(p)$}, there exists \mbox{$x'_{\params}\in{X}_{\params}$} such that 
\begin{equation}
\left\Vert\ol{\xi}_{x_{\params}p}(\tau)-x'_{\params}\right\Vert\leq\eta, \label{b04}
\end{equation}
which, by the definition of $S_\params(\Sigma)$, implies the existence of $x_{\params}\rTo^{p}_{\params}x'_{\params}$ in $S_{\params}(\Sigma)$. 
Using Lemmas \ref{lem:moment est}, the concavity of $\widehat\gamma$, the Jensen inequality \cite{oksendal}, the inequalities (\ref{supplement}), (\ref{bisim_cond}), (\ref{b02}), (\ref{b04}), and triangle inequality, we obtain
\begin{small}
\begin{align*}
\mathbb{E}\left[V(x'_{\tau},x'_{\params})\right]&=\mathbb{E}\left[V(x'_\tau,\xi_{x_{\params}p}(\tau))+V(x'_{\tau},x'_\params)-V(x'_\tau,\xi_{x_{\params}p}(\tau))\right]=  \mathbb{E}\left[V(x'_{\tau},\xi_{x_{\params}p}(\tau))\right]+\mathbb{E}\left[V(x'_{\tau},x'_\params)-V(x'_\tau,\xi_{x_{\params}p}(\tau))\right]\\\notag&\leq\underline\alpha\left(\varepsilon^q\right)\mathsf{e}^{-\kappa\tau}+\mathbb{E}\left[\widehat\gamma\left(\left\Vert\xi_{x_{\params}p}(\tau)-x'_{\params}\right\Vert\right)\right]\leq\underline\alpha\left(\varepsilon^q\right)\mathsf{e}^{-\kappa\tau}+\widehat\gamma\left(\mathbb{E}\left[\left\Vert\xi_{x_{\params}p}(\tau)-\ol{\xi}_{x_{\params}p}(\tau)+\ol{\xi}_{x_{\params}p}(\tau)-x'_{\params}\right\Vert\right]\right)
\\\notag&\leq\underline\alpha\left(\varepsilon^q\right)\mathsf{e}^{-\kappa\tau}+\widehat\gamma\left(\mathbb{E}\left[\left\Vert\xi_{x_{\params}p}(\tau)-\ol{\xi}_{x_{\params}p}(\tau)\right\Vert\right]+\left\Vert\ol{\xi}_{x_{\params}p}(\tau)-x'_{\params}\right\Vert\right)\\\notag&\leq\underline\alpha\left(\varepsilon^q\right)\mathsf{e}^{-\kappa\tau}+\widehat\gamma\left(\left(h_{[X_{\tau0}]_\eta}(\tau)\right)^{\frac{1}{q}}+\eta\right)\leq\underline\alpha\left(\varepsilon^q\right).
\end{align*}
\end{small}Therefore, we conclude that $\left(x'_{\tau},x'_{\params}\right)\in{R}$ and that condition (ii) in Definition \ref{APSR} holds. Since $X_{\tau 0}\subseteq\bigcup_{p\in[\mathbb{R}^n]_{\eta}}\mathcal{B}_{\eta}(p)$, 
for every $x_{\tau 0}\in{X_{\tau 0}}$ there always exists \mbox{$x_{\params 0}\in{X}_{\params 0}$} such that $\Vert{x_{\tau0}}-x_{\params0}\Vert\leq\eta.$ Then,
\begin{small}
\begin{align}\nonumber
\mathbb{E}\left[V({x_{\tau0}},x_{\params0})\right]=V({x_{\tau0}},x_{\params0})&\leq\overline\alpha\left(\Vert x_{\tau0}-x_{\params0}\Vert^q\right)\leq\overline\alpha\left(\eta^q\right)\leq\underline\alpha\left(\varepsilon^q\right),
\end{align}
\end{small}because of (\ref{bisim_cond1}) and since $\overline\alpha$ is a $\mathcal{K}_\infty$ function.
Hence, \mbox{$\left(x_{\tau0},x_{\params0}\right)\in{R}$} implying that \mbox{$S_{\tau}(\Sigma)\preceq^{\varepsilon}_\mathcal{S}S_{\params}(\Sigma)$}. In a similar way, we can prove that \mbox{$S_{\params}(\Sigma)\preceq^{\varepsilon}_{\mathcal{S}}S_{\tau}(\Sigma)$} by showing that $R^{-1}$ is an $\varepsilon$-approximate simulation relation from $S_\params(\Sigma)$ to $S_\tau(\Sigma)$. 
\end{proof}

Note that the results in \cite[Theorem 4.1]{girard2} for non-probabilistic models are fully recovered by the statement in Theorem \ref{main_theorem1} if the stochastic switched system $\Sigma$ is not affected by any noise, 
implying that $h_x^p(t)$ is identically zero for all $p\in\mathsf{P}$ and all $x\in\R^n$, 
and that the $\delta$-GAS-M$_q$ common Lyapunov function simply reduces to being $\delta$-GAS one. 


\subsubsection{Multiple Lyapunov functions}
If a common $\delta$-GAS-M$_q$  Lyapunov function does not exist or cannot be practically found, 
one can still attempt computing approximately bisimilar symbolic models by seeking mode-dependent Lyapunov functions and by restricting the set of switching signals using a condition on the dwell time $\tau_d=\widehat{N}\tau$ for some $\widehat{N}\in\N$. 

Consider a stochastic switched system $\Sigma_{\tau_d}$ and a pair $\mathsf{q}=(\tau,\eta)$ of quantization parameters, where $\tau$ is the sampling time and $\eta$ is the state space quantization. 
Given $\Sigma_{\tau_d}$ and $\mathsf{q}$, consider the following system:
\begin{small}
\begin{equation}\label{T3}
S_{\mathsf{q}}\left(\Sigma_{\tau_d}\right)=(X_{\mathsf{q}},X_{\params0},U_{\mathsf{q}},\rTo_{\mathsf{q}},Y_{\mathsf{q}},H_{\mathsf{q}}),
\end{equation}
\end{small}where $X_{\mathsf{q}}=[\R^n]_\eta\times\mathsf{P}\times\left\{0,\ldots,\widehat{N}-1\right\}$, $X_{\params0}=[\R^n]_\eta\times\mathsf{P}\times\left\{0\right\}$, $U_{\mathsf{q}}=\mathsf{P}$, $Y_{\mathsf{q}}=Y_\tau$, and
\begin{itemize}
\item $\left(x_{\mathsf{q}},p,i\right)\rTo_{\mathsf{q}}^{p}\left(x'_{\mathsf{q}},p',i'\right)$ if there exists a $x'_\params\in X_\params$ such that $\left\Vert\ol{\xi}_{x_{\mathsf{q}}p}(\tau)-x'_{\mathsf{q}}\right\Vert\leq\eta$ and one of the following holds:

\begin{itemize}
\item $i<\widehat{N}-1$, $p'=p$, and $i'=i+1$;

\item $i=\widehat{N}-1$, $p'=p$, and $i'=\widehat{N}-1$;

\item $i=\widehat{N}-1$, $p'\neq p$, and $i'=0$.
\end{itemize}


\item $H_{\mathsf{q}}(x_\mathsf{q},p,i)=x_\mathsf{q}$ for any $(x_\mathsf{q},p,i)\in[\R^n]_\eta\times\mathsf{P}\times\left\{0,\ldots,\widehat{N}-1\right\}$.
\end{itemize}


We present the second main result of this subsection, 
which relates the existence of multiple Lyapunov functions for a stochastic switched system to that of a symbolic model, based on the state space discretization. In order to show the next result, we assume that $f_p(0_n)=0_n$ only if $\Sigma_p$ is not affine and that $g_p(0_n)=0_{n\times\widehat{q}}$ for any $p\in\mathsf{P}$.
\begin{theorem}\label{main_theorem2}
Consider a stochastic switched system $\Sigma_{\tau_d}$. Let us assume that for any $p\in\mathsf{P}$, there exists a $\delta$-GAS-M$_q$ Lyapunov function $V_p$, of the form explained in Lemma \ref{lem:moment est}, for subsystem $\Sigma_{\tau_d,p}$. Moreover, assume that (\ref{eq0}) holds for some $\mu\geq1$. If $\tau_d>\log{\mu}/\kappa$, for $X_{\tau0}=X_0\times\mathsf{P}\times\left\{0,\ldots,\widehat{N}\right\}$, where $X_0=\R^n$, any $\varepsilon\in\R^+$, 
and any pair $\mathsf{q}=(\tau,\eta)$ of quantization parameters satisfying
\begin{small}
\begin{align}\label{bisim_cond_mul1}
\overline\alpha\left(\eta^q\right)&\leq\underline\alpha\left(\varepsilon^q\right),\\\label{bisim_cond_mul}
\widehat\gamma\left(\left(h_{[X_0]_\eta}(\tau)\right)^{\frac{1}{q}}+\eta\right)&\leq\frac{\frac{1}{\mu}-\mathsf{e}^{-\kappa\tau_d}}{1-\mathsf{e}^{-\kappa\tau_d}}\left(1-\mathsf{e}^{-\kappa\tau}\right)\underline\alpha\left(\varepsilon^q\right),
\end{align}
\end{small}we have that \mbox{$S_{\mathsf{q}}\left(\Sigma_{\tau_d}\right)\cong_{\mathcal{S}}^{\varepsilon}S_{\tau}\left(\Sigma_{\tau_d}\right)$}.
\end{theorem} 

It can be readily seen that when we are interested in the dynamics of $\Sigma_{\tau_d}$ on a compact $\mathsf{D}\subset\R^n$ of the form of a finite union of boxes, implying that $X_0=\mathsf{D}$, and for a precision $\varepsilon$, there always exists a sufficiently large value of $\tau$ and a small value of $\eta$, such that $\eta\leq\boxspan(\mathsf{D})$ and the conditions in (\ref{bisim_cond_mul1}) and (\ref{bisim_cond_mul}) are satisfied. For a given fixed sampling time $\tau$, the precision $\varepsilon$ is lower bounded by 
\begin{small}
\begin{equation}\label{lower_bound_mul}
\varepsilon\geq\left(\ul\alpha^{-1}\left(\frac{\widehat\gamma\left(\left(h_{[X_0]_\eta}(\tau)\right)^{\frac{1}{q}}\right)}{1-\mathsf{e}^{-\kappa\tau}}\cdot\frac{1-\mathsf{e}^{-\kappa\tau_d}}{\frac{1}{\mu}-\mathsf{e}^{-\kappa\tau_d}}\right)\right)^{\frac{1}{q}}.
\end{equation}
\end{small} 
The properties of the bound in (\ref{lower_bound_mul}) are analogous to those of the case of a common Lyapunov function. 

Note that $S_{\mathsf{q}}\left(\Sigma_{\tau_d}\right)$ has a countable number of states and it is finite if one is interested in the dynamics of $\Sigma_{\tau_d}$ on a compact $\mathsf{D}\subset\R^n$ which is always the case in practice.

\begin{proof}
The proof was inspired by the proof of Theorem 4.2 in \cite{girard2} for non-probabilistic switched systems. We start by proving \mbox{$S_{\tau}\left(\Sigma_{\tau_d}\right)\preceq^{\varepsilon}_\mathcal{S}S_{\params}\left(\Sigma_{\tau_d}\right)$}.
Consider the relation $R\subseteq X_{\tau}\times X_{\params}$ defined by 
{$\left(x_{\tau},p_1,i_1,x_{\params},p_2,i_2\right)\in R$}
if and only if $p_1=p_2=p$, $i_1=i_2=i$, and $\mathbb{E}\left[V_p\left(H_{\tau}(x_{\tau},p_1,i_1),H_{\params}(x_{\params},p_2,i_2)\right)\right]=\mathbb{E}\left[V_p\left(x_{\tau},x_{\params}\right)\right]\leq\delta_i$, where $\delta_0,\ldots,\delta_{\widehat{N}}$ are given recursively by\begin{small}$$\delta_0=\ul\alpha\left(\varepsilon^q\right),~~\delta_{i+1}=\mathsf{e}^{-\kappa\tau}\delta_i+\widehat\gamma\left(\left(h_{[X_0]_\eta}(\tau)\right)^{\frac{1}{q}}+\eta\right).$$\end{small} One can easily verify that
\begin{small}
\begin{align}\nonumber
\delta_i=&\mathsf{e}^{-i\kappa\tau}\ul\alpha\left(\varepsilon^q\right)+\widehat\gamma\left(\left(h_{[X_0]_\eta}(\tau)\right)^{\frac{1}{q}}+\eta\right)\frac{1-\mathsf{e}^{-i\kappa\tau}}{1-\mathsf{e}^{-\kappa\tau}}\\\label{delta_i}=&\frac{\widehat\gamma\left(\left(h_{[X_0]_\eta}(\tau)\right)^{\frac{1}{q}}+\eta\right)}{1-\mathsf{e}^{-\kappa\tau}}+\mathsf{e}^{-i\kappa\tau}\left(\ul\alpha\left(\varepsilon^q\right)-\frac{\widehat\gamma\left(\left(h_{[X_0]_\eta}(\tau)\right)^{\frac{1}{q}}+\eta\right)}{1-\mathsf{e}^{-\kappa\tau}}\right).
\end{align}
\end{small}
Since $\mu\geq1$, and from (\ref{bisim_cond_mul}), one has \begin{small}$$\widehat\gamma\left(\left(h_{[X_0]_\eta}(\tau)\right)^{\frac{1}{q}}+\eta\right)\leq(1-\mathsf{e}^{-\kappa\tau})\ul\alpha\left(\varepsilon^q\right).$$\end{small} It follows from (\ref{delta_i}) that $\delta_0\geq\delta_1\geq\cdots\geq\delta_{\widehat{N}-1}\geq\delta_{\widehat{N}}$. From (\ref{bisim_cond_mul}) and since $\tau_d=\widehat{N}\tau$, we get 
\begin{small}
\begin{align}\label{Ntozero}
\delta_{\widehat{N}}=&\mathsf{e}^{-\kappa\tau_d}\ul\alpha\left(\varepsilon^q\right)+\widehat\gamma\left(\left(h_{[X_0]_\eta}(\tau)\right)^{\frac{1}{q}}+\eta\right)\frac{1-\mathsf{e}^{-\kappa\tau_d}}{1-\mathsf{e}^{-\kappa\tau}}\leq\mathsf{e}^{-\kappa\tau_d}\ul\alpha\left(\varepsilon^q\right)+\left(\frac{1}{\mu}-\mathsf{e}^{-\kappa\tau_d}\right)\ul\alpha\left(\varepsilon^q\right)=\frac{\ul\alpha\left(\varepsilon^q\right)}{\mu}.
\end{align}
\end{small}
We can now prove that $R$ is an $\varepsilon$-approximate simulation relation from $S_{\tau}\left(\Sigma_{\tau_d}\right)$ to $S_{\params}\left(\Sigma_{\tau_d}\right)$.
Consider any \mbox{$\left(x_{\tau},p,i,x_{\params},p,i\right)\in R$}. Using the convexity assumption of $\underline\alpha_p$, and since it is a $\mathcal{K}_\infty$ function, and the Jensen inequality \cite{oksendal}, we have:
\begin{small}
\begin{align}\nonumber
&\ul\alpha\left(\mathbb{E}\left[\Vert H_\tau(x_{\tau},p,i)-H_\params(x_{\params},p,i)\Vert^q\right]\right)=\ul\alpha\left(\mathbb{E}\left[\Vert x_{\tau}-x_{\params}\Vert^q\right]\right)\\\notag&\leq\ul\alpha_p\left(\mathbb{E}\left[\Vert x_{\tau}-x_{\params}\Vert^q\right]\right)\leq\EE\left[\ul\alpha_p\left(\Vert x_{\tau}-x_{\params}\Vert^q\right)\right]\leq\mathbb{E}\left[V_p(x_{\tau},x_{\params})\right]\leq\delta_i\leq\delta_0.
\end{align}
\end{small}Therefore, we obtain
\begin{small}
$\left(\mathbb{E}\left[\Vert x_{\tau}-x_{\params}\Vert^q\right]\right)^{\frac{1}{q}}\leq\left(\underline\alpha^{-1}\left(\delta_0\right)\right)^{\frac{1}{q}}\leq\varepsilon$,
\end{small}because of $\ul\alpha\in\mathcal{K}_\infty$. Hence, condition (i) in Definition \ref{APSR} is satisfied.
Let us now show that condition (ii) in Definition
\ref{APSR} holds. 
Consider the transition $(x_{\tau},p,i)\rTo^{p}_{\tau} (x'_{\tau},p',i')$ in $S_{\tau}\left(\Sigma_{\tau_d}\right)$, where $x'_{\tau}=\xi_{x_{\tau}p}(\tau)$ $\PP$-a.s.. Since $V_p$ is a $\delta$-GAS-M$_q$ Lyapunov function for subsystem $\Sigma_p$, we have
\begin{small}
\begin{align}\label{b06}
\mathbb{E}\left[V_p(x'_{\tau},\xi_{x_{\params}p}(\tau))\right] &\leq \EE\left[V_p(x_\tau,x_q)\right] \mathsf{e}^{-\kappa\tau}\leq\mathsf{e}^{-\kappa\tau}\delta_i.
\end{align}
\end{small}
Since \mbox{$\R^n\subseteq\bigcup_{p\in[\mathbb{R}^n]_{\eta}}\mathcal{B}_{\eta}(p)$}, there exists \mbox{$x'_{\params}\in[\mathbb{R}^n]_{\eta}$} such that 
\begin{small}
\begin{equation}
\left\Vert\ol{\xi}_{x_{\params}p}(\tau)-x'_{\params}\right\Vert\leq\eta. \label{b07}
\end{equation}
\end{small}Using Lemmas \ref{lem:moment est}, the $\mathcal{K}_\infty$ function $\widehat\gamma$, the concavity of $\widehat\gamma_p$ in \eqref{supplement}, the Jensen inequality \cite{oksendal}, the inequalities (\ref{supplement}), (\ref{b06}), (\ref{b07}), and triangle inequality, we obtain
\begin{small}
\begin{align}\nonumber
\mathbb{E}\left[V_p(x'_{\tau},x'_{\params})\right]&=\mathbb{E}\left[V_p(x'_\tau,\xi_{x_{\params}p}(\tau))+V_p(x'_{\tau},x'_\params)-V_p(x'_\tau,\xi_{x_{\params}p}(\tau))\right]=  \mathbb{E}\left[V_p(x'_{\tau},\xi_{x_{\params}p}(\tau))\right]+\mathbb{E}\left[V_p(x'_{\tau},x'_\params)-V_p(x'_\tau,\xi_{x_{\params}p}(\tau))\right]\\\notag&\leq\mathsf{e}^{-\kappa\tau}\delta_i+\mathbb{E}\left[\widehat\gamma_p\left(\left\Vert\xi_{x_{\params}p}(\tau)-x'_{\params}\right\Vert\right)\right]\leq\mathsf{e}^{-\kappa\tau}\delta_i+\widehat\gamma_p\left(\mathbb{E}\left[\left\Vert\xi_{x_{\params}p}(\tau)-x'_{\params}\right\Vert\right]\right)\\\label{b08}
&\leq\mathsf{e}^{-\kappa\tau}\delta_i+\widehat\gamma\left(\mathbb{E}\left[\left\Vert\xi_{x_{\params}p}(\tau)-\ol{\xi}_{x_{\params}p}(\tau)\right\Vert\right]+\left\Vert\ol{\xi}_{x_{\params}p}(\tau)-x'_{\params}\right\Vert\right)\leq\mathsf{e}^{-\kappa\tau}\delta_i+\widehat\gamma\left(\left(h_{[X_0]_\eta}(\tau)\right)^{\frac{1}{q}}+\eta\right)=\delta_{i+1}.
\end{align}
\end{small}
We now examine three separate cases:
\begin{itemize}
\item If $i<\widehat{N}-1$, then $p'=p$, and $i'=i+1$; since, from (\ref{b08}), $\EE\left[V_p(x'_{\tau},x'_{\params})\right]\leq\delta_{i+1}$, we conclude that $(x'_{\tau},p,i+1,x'_{\params},p,i+1)\in R$; 

\item If $i=\widehat{N}-1$, and $p'=p$, then $i'=\widehat{N}-1$; from (\ref{b08}), $\EE\left[V_p(x'_{\tau},x'_{\params})\right]\leq\delta_{\widehat{N}}\leq\delta_{\widehat{N}-1}$, we conclude that $(x'_{\tau},p,\widehat{N}-1,x'_{\params},p,\widehat{N}-1)\in R$; 

\item If $i=\widehat{N}-1$, and $p'\neq{p}$, then $i'=0$; from (\ref{Ntozero}) and (\ref{b08}), $\EE\left[V_p(x'_{\tau},x'_{\params})\right]\leq\delta_{\widehat{N}}\leq\delta_0/\mu$. From (\ref{eq0}), it follows that $\EE\left[V_{p'}(x'_{\tau},x'_{\params})\right]\leq\mu \EE\left[V_p(x'_{\tau},x'_{\params})\right]\leq\delta_0$. Hence, $(x'_{\tau},p',0,x'_{\params},p',0)\in R$.
\end{itemize}
Therefore, we conclude that condition (ii) in Definition \ref{APSR} holds. Since $X_0\subseteq\bigcup_{p\in[\mathbb{R}^n]_{\eta}}\mathcal{B}_{\eta}(p)$, for every $\left(x_{\tau 0},p,0\right)\in X_{\tau0}$ there always exists $\left(x_{\params0},p,0\right)\in{X}_{\params0}$ such that $\Vert{x_{\tau0}}-x_{\params0}\Vert\leq\eta$. Then,
\begin{small}
\begin{align}\nonumber
&\mathbb{E}\left[V_p(H_\tau(x_{\tau0},p,0),H_{\mathsf{q}}(x_{\params0},p,0)\right]=V_p({x_{\tau0}},x_{\params0})\leq\ol\alpha_p\left(\left\| x_{\tau0}-x_{\params0}\right\|^q\right)\leq\overline\alpha\left(\Vert x_{\tau0}-x_{\params0}\Vert^q\right)\leq\overline\alpha\left(\eta^q\right)\leq\underline\alpha\left(\varepsilon^q\right),
\end{align}
\end{small}because of (\ref{bisim_cond_mul1}) and since $\overline\alpha$ is a $\mathcal{K}_\infty$ function.
Hence, $V_p(x_{\tau0},x_{\params0})\leq\delta_0$ and \mbox{$\left(x_{\tau0},p,0,x_{\params0},p,0\right)\in{R}$} implying that \mbox{$S_{\tau}(\Sigma_{\tau_d})\preceq^{\varepsilon}_\mathcal{S}S_{\params}(\Sigma_{\tau_d})$}. In a similar way, we can prove that \mbox{$S_{\params}\left(\Sigma_{\tau_d}\right)\preceq^{\varepsilon}_{\mathcal{S}}S_{\tau}\left(\Sigma_{\tau_d}\right)$} by showing that $R^{-1}$ is an $\varepsilon$-approximate simulation relation from $S_\params(\Sigma_{\tau_d})$ to $S_\tau(\Sigma_{\tau_d})$. 
\end{proof}

As before, Theorem \ref{main_theorem2} 
subsumes \cite[Theorem 4.2]{girard2} over non-probabilistic models. 

\subsection{Second approach}
This subsection contains the second main results of the paper providing bisimilar symbolic models without any space discretization.

\subsubsection{Common Lyapunov function}
First, we show one of the main results of this subsection on the construction of symbolic models based on the existence of a common $\delta$-GAS-M$_q$ Lyapunov function. 
We proceed by introducing two fully symbolic systems for the concrete one $\Sigma$.
Consider a stochastic switched system $\Sigma$ and a triple $\ol{\mathsf{q}}=\left(\tau,N,x_s\right)$ of parameters, where $\tau$ is the sampling time, $N\in\N$ is a \emph{temporal horizon}, and $x_s\in\R^n$ is a \emph{source state}.
Given $\Sigma$ and $\ol{\mathsf{q}}$, consider the following systems:
\begin{small}
\begin{align}\notag
S_{\ol{\mathsf{q}}}(\Sigma)&=(X_{\ol{\mathsf{q}}},X_{\ol{\mathsf{q}}0},U_{\ol{\mathsf{q}}},\rTo_{\ol{\mathsf{q}}},Y_{\ol{\mathsf{q}}},H_{\ol{\mathsf{q}}}),\\\notag\ol{S}_{\ol{\mathsf{q}}}(\Sigma)&=(X_{\ol{\mathsf{q}}},X_{\ol{\mathsf{q}}0},U_{\ol{\mathsf{q}}},\rTo_{\ol{\mathsf{q}}},Y_{\ol{\mathsf{q}}},\ol{H}_{\ol{\mathsf{q}}}),
\end{align}
\end{small}where $X_{\ol{\mathsf{q}}}=\mathsf{P}^N$, $X_{\ol{\params}0}=X_{\ol{\mathsf{q}}}$, $U_{\ol{\mathsf{q}}}=\mathsf{P}$, $Y_{\ol{\mathsf{q}}}=Y_{\tau}$, and
\begin{itemize}
\item $x_{\ol{\params}}\rTo_{\ol{\mathsf{q}}}^{p}x'_{\ol{\mathsf{q}}}$, where $x_{\ol{\params}}=(p_1,p_2,\ldots,p_N)$, if and only if $x'_{\ol{\params}}=(p_2,\ldots,p_N,p)$;
\item $H_{\ol{\mathsf{q}}}(x_{\ol{\params}})=\xi_{x_sx_{\ol{\params}}}(N\tau)$ $\left(\ol{H}_{\ol{\mathsf{q}}}(x_{\ol{\params}})=\ol\xi_{x_sx_{\ol{\params}}}(N\tau)\right)$.
\end{itemize}

Note that we have abused notation by identifying $x_{\ol{\params}}=(p_1,p_2,\ldots,p_N)$ with a switching signal obtained by the concatenation of modes $p_i$ \big(i.e. $x_{\ol{\params}}(t)=p_i$ for any $t\in[(i-1)\tau,i\tau[$\big) for $i=1,\ldots,N$. Notice that the proposed system $S_{\ol{\params}}(\Sigma)$ $\left(\text{resp.}~\ol{S}_{\ol{\params}}(\Sigma)\right)$ is symbolic and deterministic  in the sense of Definition \ref{system}. Note that $H_{\ol{\params}}$ and $\ol{H}_{\ol{\params}}$ are mappings from a non-probabilistic point $x_{\ol{\params}}$ to the random variable $\xi_{x_sx_{\ol{\params}}}(N\tau)$ and to the one with a Dirac probability distribution centered at $\ol\xi_{x_sx_{\ol{\params}}}(N\tau)$, respectively. One can readily verify that the transition relation of $S_{\ol{\mathsf{q}}}(\Sigma)$ (resp. $\ol{S}_{\ol{\mathsf{q}}}(\Sigma)$) admits a very compact representation under the form of a shift operator and such symbolic systems do not require any continuous space discretization.

Before providing the main results, we need the following technical lemmas.

\begin{lemma}\label{lemma1}
Consider a stochastic switched system $\Sigma$, admitting a common $\delta$-GAS-M$_q$ Lyapunov function $V$, and consider its corresponding symbolic model $\ol{S}_{\ol{\params}}(\Sigma)$. We have:
\begin{small}
\begin{align}\label{upper_bound}
\ol\eta\leq&\left(\ul\alpha^{-1}\left(\mathsf{e}^{-\kappa N\tau}\max_{p\in \mathsf{P}}V\left(\ol\xi_{x_sp}(\tau),x_s\right)\right)\right)^{1/q},
\end{align}
\end{small}
where
\begin{small}
\begin{align}\label{eta}
\ol\eta\Let\max_{\substack{p\in\mathsf{P},x_{\ol{\params}}\in X_{\ol{\params}}\\x_{\ol{\params}}\rTo_{\ol{\params}}^{p}x'_{\ol{\params}}}}\left\Vert\ol\xi_{\ol{H}_{\ol{\params}}(x_{\ol{\params}})p}(\tau)-\ol{H}_{\ol{\params}}\left(x'_{\ol{\params}}\right)\right\Vert.
\end{align}
\end{small}
\end{lemma}

\begin{proof}
Let $x_{\ol{\params}}\in X_{\ol{\params}}$, where $x_{\ol{\params}}=\left(p_1,p_2,\ldots,p_N\right)$, and $p\in U_{\ol{\params}}$. Using the definition of $\ol{S}_{\ol{\params}}(\Sigma)$, one obtains $x_{\ol{\params}}\rTo_{\ol{\params}}^{p}x'_{\ol{\params}}$, where $x'_{\ol{\params}}=\left(p_2,\ldots,p_N,p\right)$. Since $V$ is a common $\delta$-GAS-M$_q$ Lyapunov function for $\Sigma$ and based on the manipulations in the proof of Theorem \ref{theorem2}, we have:
\begin{small}
\begin{align}\notag
\ul\alpha\left(\left\Vert\ol\xi_{\ol{H}_{\ol{\params}}(x_{\ol{\params}})p}(\tau)-\ol{H}_{\ol{\params}}\left(x'_{\ol{\params}}\right)\right\Vert^q\right)&\leq V\left(\ol\xi_{\ol{H}_{\ol{\params}}(x_{\ol{\params}})p}(\tau),\ol{H}_{\ol{\params}}\left(x'_{\ol{\params}}\right)\right)=V\left(\ol\xi_{\ol\xi_{x_sx_{\ol{\params}}}(N\tau)p}(\tau),\ol\xi_{x_sx'_{\ol{\params}}}(N\tau)\right)\\\notag&=V\left(\ol\xi_{\ol\xi_{x_sp_1}(\tau)(p_2,\ldots,p_N,p)}(N\tau),\ol\xi_{x_s(p_2,\ldots,p_N,p)}(N\tau)\right)\leq\mathsf{e}^{-\kappa N\tau}V\left(\ol\xi_{x_sp_1}(\tau),x_s\right).
\end{align}
\end{small}
Hence, one gets
\begin{small}
\begin{align}\label{upper_bound2}
\left\Vert\ol\xi_{\ol{H}_{\ol{\params}}(x_{\ol{\params}})p}(\tau)-\ol{H}_{\ol{\params}}\left(x'_{\ol{\params}}\right)\right\Vert\leq\left(\ul\alpha^{-1}\left(\mathsf{e}^{-\kappa N\tau}V\left(\ol\xi_{x_sp_1}(\tau),x_s\right)\right)\right)^{1/q},
\end{align}
\end{small}because of $\ul\alpha\in\mathcal{K}_\infty$. Since the inequality \eqref{upper_bound2} holds for all $x_{\ol{\params}}\in X_{\ol{\params}}$ and $p\in U_{\ol{\params}}$, and $\ul\alpha\in\mathcal{K}_\infty$, inequality \eqref{upper_bound} holds. 
\end{proof}

The next lemma provides a similar result as the one of Lemma \ref{lemma1}, but by using the symbolic model $S_{\ol{\params}}(\Sigma)$ rather than $\ol{S}_{\ol{\params}}(\Sigma)$.
\begin{lemma}\label{lemma4}
Consider a stochastic switched system $\Sigma$, admitting a common $\delta$-GAS-M$_q$ Lyapunov function $V$, and consider its corresponding symbolic model $S_{\ol{\params}}(\Sigma)$. One has:
\begin{small}
\begin{align}\label{upper_bound4}
\widehat\eta\leq&\left(\ul\alpha^{-1}\left(\mathsf{e}^{-\kappa N\tau}\max_{p\in\mathsf{P}}\EE\left[V\left(\xi_{x_sp}(\tau),x_s\right)\right]\right)\right)^{1/q},
\end{align}
\end{small}
where
\begin{small}
\begin{align}\label{eta1}
\widehat\eta\Let\max_{\substack{p\in\mathsf{P},x_{\ol{\params}}\in X_{\ol{\params}}\\x_{\ol{\params}}\rTo_{\ol{\params}}^px'_{\ol{\params}}}}\EE\left[\left\Vert\xi_{H_{\ol{\params}}(x_{\ol{\params}})p}(\tau)-H_{\ol{\params}}\left(x'_{\ol{\params}}\right)\right\Vert\right].
\end{align}
\end{small}
\end{lemma}

\begin{proof}
The proof is similar to the one of Lemma \ref{lemma1} and can be shown by using convexity of $\ul\alpha$ and Jensen inequality \cite{oksendal}. 
\end{proof}

We can now present the first main result of this subsection, relating the existence of a common $\delta$-GAS-M$_q$ Lyapunov function to the construction of a bisimilar finite abstraction without any continuous space discretization. In order to show the next result, we assume that $f_p(0_n)=0_n$ only if $\Sigma_p$ is not affine and $g_p(0_n)=0_{n\times\widehat{q}}$ for any $p\in\mathsf{P}$.

\begin{theorem}\label{main_theorem}
Consider a stochastic switched system $\Sigma$ admitting a common $\delta$-GAS-M$_q$ Lyapunov function $V$, of the form of the one explained in Lemma \ref{lem:moment est}. Let $\ol\eta$ be given by \eqref{eta}. For any $\varepsilon\in\R^+$ and any triple $\ol{\mathsf{q}}=\left(\tau,N,x_s\right)$ of parameters satisfying
\begin{small}
\begin{align}
\label{bisim_cond11}
\mathsf{e}^{-\kappa\tau}\underline\alpha\left(\varepsilon^q\right)+\widehat\gamma\left(\left(h_{x_s}((N+1)\tau)\right)^{\frac{1}{q}}+\ol\eta\right)&\leq\underline\alpha\left(\varepsilon^q\right),
\end{align}
\end{small}the relation \begin{small}$$R=\left\{\left(x_\tau,x_{\ol{\params}}\right)\in X_\tau\times X_{\ol{\params}}\,\,|\,\,\EE\left[V\left(x_\tau,\ol{H}_{\ol{\params}}(x_{\ol{\params}})\right)\right]\leq\ul\alpha\left(\varepsilon^q\right)\right\}$$\end{small}is an $\varepsilon$-approximate bisimulation relation between $\ol{S}_{\ol{\mathsf{q}}}(\Sigma)$ and $S_{\tau}(\Sigma)$.
\end{theorem}

\begin{proof}
We start by proving that $R$ is an $\varepsilon$-approximate simulation relation from $S_{\tau}(\Sigma)$ to $\ol{S}_{{\ol{\params}}}(\Sigma)$. Consider any \mbox{$\left(x_{\tau},x_{{\ol{\params}}}\right)\in R$}. Condition (i) in Definition \ref{APSR} is satisfied because
\begin{small}
\begin{equation}
\label{convexity}
(\mathbb{E}[\Vert x_{\tau}-\ol{H}_{\ol{\params}}(x_{{\ol{\params}}})\Vert^q])^{\frac{1}{q}}\leq\left(\underline\alpha^{-1}\left(\mathbb{E}\left[V\left(x_{\tau},\ol{H}_{\ol{\params}}(x_{{\ol{\params}}})\right)\right]\right)\right)^{\frac{1}{q}}\leq\varepsilon.
\end{equation}
\end{small}We used the convexity assumption of $\underline\alpha$ and the Jensen inequality \cite{oksendal} to show the inequalities in (\ref{convexity}). Let us now show that condition (ii) in Definition
\ref{APSR} holds. Consider the transition \mbox{$x_{\tau}\rTo^{p}_{\tau} x'_{\tau}=\xi_{x_{\tau}p}(\tau)$} $\PP$-a.s. in $S_{\tau}(\Sigma)$. Since $V$ is a common $\delta$-GAS-M$_q$ Lyapunov function for $\Sigma$, we have (cf. proof of Theorem \ref{theorem2})
\begin{small}
\begin{align}\label{b020}
\mathbb{E}[V(x'_{\tau},\xi_{\ol{H}_{\ol{\params}}(x_{{\ol{\params}}})p}(\tau))] \leq \EE[V(x_\tau,\ol{H}_{\ol{\params}}(x_q))] \mathsf{e}^{-\kappa\tau}\leq \underline\alpha(\varepsilon^q) \mathsf{e}^{-\kappa\tau}.
\end{align}
\end{small}Note that, by the definition of $\ol{S}_{\ol{\params}}(\Sigma)$, there exists $x_{{\ol{\params}}}\rTo^{p}_{{\ol{\params}}}x'_{{\ol{\params}}}$ in $\ol{S}_{{\ol{\params}}}(\Sigma)$.
Using Lemma \ref{lem:moment est}, the concavity of $\widehat\gamma$, the Jensen inequality \cite{oksendal}, equation \eqref{eta}, the inequalities (\ref{supplement}), (\ref{bisim_cond11}), (\ref{b020}), and triangle inequality, we obtain
\begin{small}
\begin{align*}
\mathbb{E}[V(x'_{\tau},\ol{H}_{\ol{\params}}(x'_{{\ol{\params}}}))]=&\mathbb{E}[V(x'_\tau,\xi_{\ol{H}_{\ol{\params}}(x_{{\ol{\params}}})p}(\tau))+V(x'_{\tau},\ol{H}_{\ol{\params}}(x'_{\ol{\params}}))-V(x'_\tau,\xi_{\ol{H}_{\ol{\params}}(x_{{\ol{\params}}})p}(\tau))]\\ \notag
=&  \mathbb{E}[V(x'_{\tau},\xi_{\ol{H}_{\ol{\params}}(x_{{\ol{\params}}})p}(\tau))]+\mathbb{E}[V(x'_{\tau},\ol{H}_{\ol{\params}}(x'_{\ol{\params}}))-V(x'_\tau,\xi_{\ol{H}_{\ol{\params}}(x_{{\ol{\params}}})p}(\tau))]\\\notag\leq&\underline\alpha(\varepsilon^q)\mathsf{e}^{-\kappa\tau}+\mathbb{E}[\widehat\gamma(\Vert\xi_{\ol{H}_{\ol{\params}}(x_{{\ol{\params}}})p}(\tau)-\ol{H}_{\ol{\params}}(x'_{{\ol{\params}}})\Vert)]\\\notag
\leq&\underline\alpha(\varepsilon^q)\mathsf{e}^{-\kappa\tau}+\widehat\gamma(\mathbb{E}[\Vert\xi_{\ol{H}_{\ol{\params}}(x_{{\ol{\params}}})p}(\tau)-\ol{\xi}_{\ol{H}_{\ol{\params}}(x_{{\ol{\params}}})p}(\tau)+\ol{\xi}_{\ol{H}_{\ol{\params}}(x_{{\ol{\params}}})p}(\tau)-\ol{H}_{\ol{\params}}(x'_{{\ol{\params}}})\Vert])\\\notag
\leq&\underline\alpha(\varepsilon^q)\mathsf{e}^{-\kappa\tau}+\widehat\gamma(\mathbb{E}[\Vert\xi_{\ol{H}_{\ol{\params}}(x_{{\ol{\params}}})p}(\tau)-\ol{\xi}_{\ol{H}_{\ol{\params}}(x_{{\ol{\params}}})p}(\tau)\Vert]+\Vert\ol{\xi}_{\ol{H}_{\ol{\params}}(x_{{\ol{\params}}})p}(\tau)-\ol{H}_{\ol{\params}}(x'_{{\ol{\params}}})\Vert)\\\notag
\leq&\underline\alpha(\varepsilon^q)\mathsf{e}^{-\kappa\tau}+\widehat\gamma((h_{x_s}((N+1)\tau))^{\frac{1}{q}}+\eta)\leq\underline\alpha(\varepsilon^q).
\end{align*}
\end{small}Therefore, we conclude that \mbox{$\left(x'_{\tau},x'_{{\ol{\params}}}\right)\in{R}$} and that condition (ii) in Definition \ref{APSR} holds.

In a similar way, we can prove that that $R^{-1}$ is an
$\varepsilon$-approximate simulation relation from $\ol{S}_{{\ol{\params}}}(\Sigma)$ to $S_{\tau}(\Sigma)$ implying that $R$ is an $\varepsilon$-approximate bisimulation relation between $\ol{S}_{{\ol{\params}}}(\Sigma)$ and $S_\tau(\Sigma)$.
\end{proof}

Note that one can also use any over approximation of $\ol\eta$ such as the one in \eqref{upper_bound} instead of $\ol\eta$ in condition \eqref{bisim_cond11}. By choosing $N$ sufficiently large, one can enforce $h_{x_s}((N+1)\tau)$ and $\ol\eta$ to be sufficiently small. Hence, it can be readily seen that for a given precision $\varepsilon$,
there always exists a large value of $N$, such that the condition in (\ref{bisim_cond11}) is satisfied.

Note that the results in \cite{corronc} for non-probabilistic models are fully recovered by the statement in Theorem \ref{main_theorem} if $\Sigma$ is not affected by any noise.

The next theorem provides a result that is similar to the one of Theorem \ref{main_theorem}, but by using the symbolic model $S_{{\ol{\params}}}(\Sigma)$.

\begin{theorem}\label{main_theorem3}
Consider a stochastic switched system $\Sigma$, admitting a common $\delta$-GAS-M$_q$ Lyapunov function $V$. Let $\widehat\eta$ be given by \eqref{eta1}. For any $\varepsilon\in\R^+$ and any triple $\ol{\mathsf{q}}=\left(\tau,N,x_s\right)$ of parameters satisfying
\begin{small}
\begin{align}
\label{bisim_cond3}
\mathsf{e}^{-\kappa\tau}\underline\alpha\left(\varepsilon^q\right)+\widehat\gamma\left(\widehat\eta\right)&\leq\underline\alpha\left(\varepsilon^q\right),
\end{align}
\end{small}
the relation \begin{small}$$R=\left\{(x_\tau,x_{\ol{\params}})\in X_\tau\times X_{\ol{\params}}\,\,|\,\,\EE\left[V(x_\tau,H_{\ol{\params}}(x_{\ol{\params}}))\right]\leq\ul\alpha\left(\varepsilon^q\right)\right\}$$\end{small}is an $\varepsilon$-approximate bisimulation relation between ${S}_{\ol{\mathsf{q}}}(\Sigma)$ and $S_{\tau}(\Sigma)$.
\end{theorem}

\begin{proof}
The proof is similar to the one of Theorem \ref{main_theorem}.
\end{proof}

Here, one can also use any over approximation of $\widehat\eta$ such as the one in \eqref{upper_bound4} instead of $\widehat\eta$ in condition \eqref{bisim_cond3}. Finally, we establish the results on the existence of symbolic model $\ol{S}_{\ol{\params}}(\Sigma)$ (resp. $S_{\ol{\params}}(\Sigma)$) such that \mbox{$\ol{S}_{\ol{\params}}(\Sigma)\cong_{\mathcal{S}}^{\varepsilon}S_\tau(\Sigma)$} (resp. \mbox{$S_{\ol{\params}}(\Sigma)\cong_{\mathcal{S}}^{\varepsilon}S_\tau(\Sigma)$}).

\begin{theorem}\label{main_theorem5}
Consider the result in Theorem \ref{main_theorem}. If we choose: \begin{small}$$X_{\tau0}=\{x\in\R^n\,\,|\,\,\Vert x-\ol{H}_{\ol{\params}}(x_{{\ol{\params}}0})\Vert\leq\left(\ol\alpha^{-1}\left(\ul\alpha\left(\varepsilon^q\right)\right)\right)^{\frac{1}{q}},~\forall x_{{\ol{\params}}0}\in X_{{\ol{\params}}0}\},$$\end{small}then we have \mbox{$\ol{S}_{\ol{\params}}(\Sigma)\cong_{\mathcal{S}}^{\varepsilon}S_\tau(\Sigma)$}.
\end{theorem}

\begin{proof}
We start by proving that \mbox{$S_{\tau}(\Sigma)\preceq^{\varepsilon}_\mathcal{S}\ol{S}_{{\ol{\params}}}(\Sigma)$}. For every $x_{\tau 0}\in{X_{\tau 0}}$, there always exists \mbox{$x_{{\ol{\params}} 0}\in{X}_{{\ol{\params}} 0}$} such that $\Vert{x_{\tau0}}-\ol{H}_{\ol{\params}}(x_{{\ol{\params}}0})\Vert\leq\left(\ol\alpha^{-1}\left(\ul\alpha\left(\varepsilon^q\right)\right)\right)^{\frac{1}{q}}$. Then,
\begin{small}
\begin{align}\nonumber
\mathbb{E}\left[V\left({x_{\tau0}},\ol{H}_{\ol{\params}}(x_{{\ol{\params}}0})\right)\right]&=V\left({x_{\tau0}},\ol{H}_{\ol{\params}}(x_{{\ol{\params}}0})\right)\leq\overline\alpha\left(\left\Vert x_{\tau0}-\ol{H}_{\ol{\params}}(x_{{\ol{\params}}0})\right\Vert^q\right)\leq\underline\alpha\left(\varepsilon^q\right),
\end{align}
\end{small}since $\overline\alpha$ is a $\mathcal{K}_\infty$ function.
Hence, \mbox{$\left(x_{\tau0},x_{{\ol{\params}}0}\right)\in{R}$} implying that \mbox{$S_{\tau}(\Sigma)\preceq^{\varepsilon}_\mathcal{S}\ol{S}_{{\ol{\params}}}(\Sigma)$}. In a similar way, we can show that \mbox{$\ol{S}_{{\ol{\params}}}(\Sigma)\preceq^{\varepsilon}_{\mathcal{S}}S_{\tau}(\Sigma)$}, equipped with the relation $R^{-1}$, which completes the proof.
\end{proof}

The next theorem provides a similar result as the one of Theorem \ref{main_theorem5}, but by using the symbolic model $S_{\ol{\params}}(\Sigma)$.

\begin{theorem}\label{main_theorem6}
Consider the results in Theorem \ref{main_theorem3}. If we choose: \begin{small}
\begin{align}\notag
X_{\tau0}=\{&a\in \mathcal{X}_0\,\,|\,\,\left(\EE\left[\left\Vert a-H_{\ol{\params}}(x_{{\ol{\params}}0})\right\Vert^q\right]\right)^{\frac{1}{q}}\leq\left(\ol\alpha^{-1}\left(\ul\alpha\left(\varepsilon^q\right)\right)\right)^{\frac{1}{q}},~\forall x_{{\ol{\params}}0}\in X_{{\ol{\params}}0}\},
\end{align}
\end{small}then we have \mbox{$S_{\ol{\params}}(\Sigma)\cong_{\mathcal{S}}^{\varepsilon}S_\tau(\Sigma)$}.
\end{theorem}

\begin{proof}
The proof is similar to the one of Theorem \ref{main_theorem5}.
\end{proof}

\subsubsection{Multiple Lyapunov functions}
Here, we provide results on the construction of symbolic models for $\Sigma_{\tau_d}$ without any continuous space discretization.
Consider a stochastic switched system $\Sigma_{\tau_d}$ and a triple $\ol{\mathsf{q}}=\left(\tau,N,x_s\right)$ of parameters.
Given $\Sigma_{\tau_d}$ and $\ol{\mathsf{q}}$, consider the following systems:
\begin{small}
\begin{align}\notag
S_{\ol{\mathsf{q}}}(\Sigma_{\tau_d})&=(X_{\ol{\mathsf{q}}},X_{\ol{\mathsf{q}}0},U_{\ol{\mathsf{q}}},\rTo_{\ol{\mathsf{q}}},Y_{\ol{\mathsf{q}}},H_{\ol{\mathsf{q}}}),\\\notag\ol{S}_{\ol{\mathsf{q}}}(\Sigma_{\tau_d})&=(X_{\ol{\mathsf{q}}},X_{\ol{\mathsf{q}}0},U_{\ol{\mathsf{q}}},\rTo_{\ol{\mathsf{q}}},Y_{\ol{\mathsf{q}}},\ol{H}_{\ol{\mathsf{q}}}),
\end{align}
\end{small}consisting of: $X_{\ol{\mathsf{q}}}=\mathsf{P}^{N}\times\{0,\ldots,\widehat{N}-1\}$, $U_{\ol{\mathsf{q}}}=\mathsf{P}$, $Y_{\ol{\mathsf{q}}}=Y_\tau$, and
\begin{small}
\begin{itemize}
\item \begin{small}
\begin{itemize}
\item if $N\leq\widehat{N}-1$: $X_{{\ol{\params}}0}=\left\{\left(p,\ldots,p,N\right)\,\,|\,\,\forall p\in\mathsf{P}\right\}$; 
\item if $N>\widehat{N}-1$: $X_{{\ol{\params}}0}=\{(\overbrace{p_1,\ldots,p_1}^{m_1~\text{times}},\dots,\overbrace{p_k,\ldots,p_k}^{m_k~\text{times}},i)|~~\exists k\in\N~~\text{s.t.}~~ m_1,\ldots,m_{k-1}\geq{\widehat{N}},~i=\min\{m_k-1,\widehat{N}-1\},~p_1,\ldots,p_k\in\mathsf{P}\}$;
\end{itemize}\end{small}
\item $\left(p_1,p_2,\ldots,p_N,i\right)\rTo_{\ol{\mathsf{q}}}^{p_N}\left(p_2,\ldots,p_N,p,i'\right)$ if one of the following holds:

\begin{itemize}
\item $i<\widehat{N}-1$, $p=p_N$, and $i'=i+1$;
\item $i=\widehat{N}-1$, $p=p_N$, and $i'=\widehat{N}-1$;
\item $i=\widehat{N}-1$, $p\neq p_N$, and $i'=0$;
\end{itemize}
\item $H_{\ol{\mathsf{q}}}(x_{\ol{\params}},i)=\xi_{x_sx_{\ol{\params}}}(N\tau)$ $\left(\ol{H}_{\ol{\mathsf{q}}}(x_{\ol{\params}},i)=\ol\xi_{x_sx_{\ol{\params}}}(N\tau)\right)$ for any $(x_{\ol{\params}},i)\in X_{\ol{\params}}$, where $x_{\ol{\params}}=\left(p_1,\ldots,p_N\right)$.
\end{itemize}
\end{small}

Notice that the proposed system $S_{\ol{\params}}(\Sigma_{\tau_d})$ $\left(\text{resp.}~\ol{S}_{{\ol{\params}}}(\Sigma_{\tau_d})\right)$ is symbolic and deterministic in the sense of Definition \ref{system}. Note that the set $X_{{\ol{\params}}0}$ is chosen in such a way that it respects the dwell time of switching signals (i.e. being in each mode at least $\tau_d=\widehat{N}\tau$ seconds).

Before providing the second main result of this subsection, we need the following technical results, similar to the ones in Lemmas \ref{lemma1} and \ref{lemma4}.

\begin{lemma}\label{lemma11}
Consider a stochastic switched system $\Sigma_{\tau_d}$, admitting multiple $\delta$-GAS-M$_q$ Lyapunov functions $V_p$, and consider its corresponding symbolic model $\ol{S}_{\ol{\params}}(\Sigma_{\tau_d})$. Moreover, assume that (\ref{eq0}) holds for some $\mu\geq1$. If $\tau_d>\log{\mu}/\kappa$, then we have:
\begin{small}
\begin{align}
\label{upper_bound5}
\ol\eta\leq&\left(\ul\alpha^{-1}(\mathsf{e}^{-(\kappa-\log\mu/\tau_d) N\tau}\max_{p,p'\in \mathsf{P}}V_{p'}(\ol\xi_{x_sp}(\tau),x_s))\right)^{1/q},
\end{align}
\end{small}
where
\begin{small}
\begin{align}\label{eta2}
\ol\eta\Let\max_{\substack{(x_{\ol{\params}},i)\in X_{\ol{\params}}\\(x_{\ol{\params}},i)\rTo^p_{\ol{\params}}(x'_{\ol{\params}},i')}}\Vert\ol\xi_{\ol{H}_{\ol{\params}}(x_{\ol{\params}},i)p}(\tau)-\ol{H}_{\ol{\params}}\left(x'_{\ol{\params}},i'\right)\Vert.
\end{align}
\end{small}
\end{lemma}

The proof is similar to the proof of Lemma \ref{lemma1}.

The next lemma provides a similar result as the one of Lemma \ref{lemma11}, but by using the symbolic model $S_{\ol{\params}}(\Sigma_{\tau_d})$ rather than $\ol{S}_{\ol{\params}}(\Sigma_{\tau_d})$.
\begin{lemma}\label{lemma44}
Consider a stochastic switched system $\Sigma_{\tau_d}$, admitting multiple $\delta$-GAS-M$_q$ Lyapunov functions $V_p$, and consider its corresponding symbolic model ${S}_{\ol{\params}}(\Sigma_{\tau_d})$. Moreover, assume that (\ref{eq0}) holds for some $\mu\geq1$. If $\tau_d>\log{\mu}/\kappa$, then we have:
\begin{small}
\begin{align}\label{upper_bound6}
\widehat\eta\leq&(\ul\alpha^{-1}(\mathsf{e}^{-(\kappa-\log\mu/\tau_d) N\tau}\max_{p,p'\in\mathsf{P}}\EE[V_{p'}\left(\xi_{x_sp}(\tau),x_s\right)]))^{\frac{1}{q}},
\end{align}
\end{small}
where
\begin{small}
\begin{align}\label{eta3}
\widehat\eta\Let\max_{\substack{(x_{\ol{\params}},i)\in X_{\ol{\params}}\\(x_{\ol{\params}},i)\rTo^p_{\ol{\params}}(x'_{\ol{\params}},i')}}\EE[\Vert\xi_{H_{\ol{\params}}(x_{\ol{\params}},i)p}(\tau)-H_{\ol{\params}}\left(x'_{\ol{\params}},i'\right)\Vert].
\end{align}
\end{small}
\end{lemma}
The proof is similar to the proof of Lemma \ref{lemma4}.

Now, we present the second main result of this subsection, relating the existence of multiple Lyapunov functions to that of a bisimilar finite abstractions without any continuous space discretization. In order to show the next result, we assume that $f_p(0_n)=0_n$ only if $\Sigma_{\tau_d,p}$ is not affine and $g_p(0_n)=0_{n\times\widehat{q}}$ for any $p\in\mathsf{P}$.

\begin{theorem}\label{main_theorem33}
Consider a stochastic switched system $\Sigma_{\tau_d}$. Let us assume that for any $p\in\mathsf{P}$, there exists a $\delta$-GAS-M$_q$ Lyapunov function $V_p$, of the form of the one explained in Lemma \ref{lem:moment est}, for subsystem $\Sigma_{\tau_d,p}$. Moreover, assume that (\ref{eq0}) holds for some $\mu\geq1$. Let $\ol\eta$ be given by \eqref{eta2}. If $\tau_d>\log{\mu}/\kappa$, for any $\varepsilon\in\R^+$, 
and any triple $\ol{\mathsf{q}}=\left(\tau,N,x_s\right)$ of parameters satisfying
\begin{small}
\begin{align}
\label{bisim_cond_mul0}
\widehat\gamma(\left(h_{x_s}\left((N+1)\tau\right)\right)^{\frac{1}{q}}+\ol\eta)&\leq\frac{\frac{1}{\mu}-\mathsf{e}^{-\kappa\tau_d}}{1-\mathsf{e}^{-\kappa\tau_d}}(1-\mathsf{e}^{-\kappa\tau})\underline\alpha(\varepsilon^q),
\end{align}
\end{small}there exists an $\varepsilon$-approximate bisimulation relation $R$ between $\ol{S}_{\ol{\mathsf{q}}}(\Sigma_{\tau_d})$ and $S_{\tau}(\Sigma_{\tau_d})$ as the following:\\ $\left(x_{\tau},p_1,i_1,x_{{\ol{\params}}},i_2\right)\in R$, where $x_{\ol{\params}}=(\ol{p}_1,\ldots,\ol{p}_N)$, if and only if $p_1=\ol{p}_N=p$, $i_1=i_2=i$, and \begin{small}$$\mathbb{E}[V_p(H_{\tau}(x_{\tau},p_1,i_1),\ol{H}_{{\ol{\params}}}(x_{{\ol{\params}}},i_2))]=\mathbb{E}[V_p(x_{\tau},\ol\xi_{x_sx_{\ol{\params}}}(N\tau))]\leq\delta_i,$$\end{small}where $\delta_0,\ldots,\delta_{\widehat{N}-1}$ are given recursively by \begin{small}$\delta_{i+1}=\mathsf{e}^{-\kappa\tau}\delta_i+\widehat\gamma\left(\left(h_{x_s}\left((N+1)\tau\right)\right)^{\frac{1}{q}}+\ol\eta\right)$\end{small} and $\delta_0=\ul\alpha\left(\varepsilon^q\right)$.
\end{theorem} 

\begin{proof}
Consider the relation $R\subseteq X_{\tau}\times X_{{\ol{\params}}}$ defined by $\left(x_{\tau},p_1,i_1,x_{{\ol{\params}}},i_2\right)\in R$, where $x_{\ol{\params}}=(\ol{p}_1,\ldots,\ol{p}_N)$, if and only if $p_1=\ol{p}_N=p$, $i_1=i_2=i$, and \begin{small}$$\mathbb{E}[V_p(H_{\tau}(x_{\tau},p_1,i_1),\ol{H}_{{\ol{\params}}}(x_{{\ol{\params}}},i_2))]=\mathbb{E}[V_p(x_{\tau},\ol\xi_{x_sx_{\ol{\params}}}(N\tau))]\leq\delta_i,$$\end{small}where $\delta_0,\ldots,\delta_{\widehat{N}}$ are given recursively by \begin{small}$$\delta_0=\ul\alpha\left(\varepsilon^q\right),~~\delta_{i+1}=\mathsf{e}^{-\kappa\tau}\delta_i+\widehat\gamma\left(\left(h_{x_s}\left((N+1)\tau\right)\right)^{\frac{1}{q}}+\eta\right).$$\end{small}One can easily verify that
\begin{small}
\begin{align}\label{delta_i}
\delta_i&=\mathsf{e}^{-i\kappa\tau}\ul\alpha(\varepsilon^q)+\widehat\gamma((h_{x_s}((N+1)\tau))^{\frac{1}{q}}+\eta)\frac{1-\mathsf{e}^{-i\kappa\tau}}{1-\mathsf{e}^{-\kappa\tau}}\\\notag&=\frac{\widehat\gamma((h_{x_s}((N+1)\tau))^{\frac{1}{q}}+\eta)}{1-\mathsf{e}^{-\kappa\tau}}+\mathsf{e}^{-i\kappa\tau}(\ul\alpha(\varepsilon^q)-\frac{\widehat\gamma((h_{x_s}((N+1)\tau))^{\frac{1}{q}}+\eta)}{1-\mathsf{e}^{-\kappa\tau}}).
\end{align}
\end{small}
Since $\mu\geq1$, and from (\ref{bisim_cond_mul0}), one has \begin{small}$$\widehat\gamma\left(\left(h_{x_s}\left((N+1)\tau\right)\right)^{\frac{1}{q}}+\eta\right)\leq(1-\mathsf{e}^{-\kappa\tau})\ul\alpha\left(\varepsilon^q\right).$$\end{small}It follows from (\ref{delta_i}) that $\delta_0\geq\delta_2\geq\cdots\geq\delta_{\widehat{N}-1}\geq\delta_{\widehat{N}}$. From (\ref{bisim_cond_mul0}) and since $\tau_d=\widehat{N}\tau$, we get 
\begin{small}
\begin{align}\label{Ntozero}
\delta_{\widehat{N}}=&\mathsf{e}^{-\kappa\tau_d}\ul\alpha(\varepsilon^q)+\widehat\gamma((h_{x_s}((N+1)\tau))^{\frac{1}{q}}+\eta)\frac{1-\mathsf{e}^{-\kappa\tau_d}}{1-\mathsf{e}^{-\kappa\tau}}\leq\mathsf{e}^{-\kappa\tau_d}\ul\alpha(\varepsilon^q)+(\frac{1}{\mu}-\mathsf{e}^{-\kappa\tau_d})\ul\alpha(\varepsilon^q)=\frac{\ul\alpha(\varepsilon^q)}{\mu}.
\end{align}
\end{small}We start by proving that $R$ is an $\varepsilon$-approximate simulation relation from $S_{\tau}(\Sigma_{\tau_d})$ to $\ol{S}_{{\ol{\params}}}(\Sigma_{\tau_d})$. Consider any \mbox{$\left(x_{\tau},p,i,x_{{\ol{\params}}},i\right)\in R$}. Using the convexity assumption of $\underline\alpha_p$, and since it is a $\mathcal{K}_\infty$ function, and the Jensen inequality \cite{oksendal}, we have:
\begin{small}
\begin{align}\nonumber
\ul\alpha(\mathbb{E}[\Vert H_\tau(x_{\tau},p,i)-H_{\ol{\params}}(x_{{\ol{\params}}},i)\Vert^q])&=\ul\alpha(\mathbb{E}[\Vert x_{\tau}-\ol\xi_{x_sx_{\ol{\params}}}(N\tau)\Vert^q])\leq\ul\alpha_p(\mathbb{E}[\Vert x_{\tau}-\ol\xi_{x_sx_{\ol{\params}}}(N\tau)\Vert^q])\leq\EE[\ul\alpha_p(\Vert x_{\tau}-\ol\xi_{x_sx_{\ol{\params}}}(N\tau)\Vert^q)]\\\notag&\leq\mathbb{E}[V_p(x_{\tau},\ol\xi_{x_sx_{\ol{\params}}}(N\tau))]\leq\delta_i\leq\delta_0.
\end{align}
\end{small}Therefore, we obtain
\begin{small}
$(\mathbb{E}[\Vert x_{\tau}-\ol\xi_{x_sx_{\ol{\params}}}(N\tau)\Vert^q])^{\frac{1}{q}}\leq(\underline\alpha^{-1}(\delta_0))^{\frac{1}{q}}\leq\varepsilon$,
\end{small}
because of $\ul\alpha\in\mathcal{K}_\infty$. Hence, condition (i) in Definition \ref{APSR} is satisfied.
Let us now show that condition (ii) in Definition
\ref{APSR} holds. 
Consider the transition $(x_{\tau},p,i)\rTo^{p}_{\tau} (x'_{\tau},p',i')$ in $S_{\tau}\left(\Sigma_{\tau_d}\right)$, where $x'_{\tau}=\xi_{x_{\tau}p}(\tau)$ $\PP$-a.s.. Since $V_p$ is a $\delta$-GAS-M$_q$ Lyapunov function for subsystem $\Sigma_p$, we have
\begin{small}
\begin{align}\label{b06}
\mathbb{E}[V_{p}(x'_{\tau},\xi_{\ol{H}_{\ol{\params}}(x_{{\ol{\params}}},i)p}(\tau))] &\leq \EE[V_{p}(x_\tau,\ol{H}_{\ol{\params}}(x_q,i))] \mathsf{e}^{-\kappa\tau}\leq\mathsf{e}^{-\kappa\tau}\delta_i.
\end{align}
\end{small}Using Lemma \ref{lem:moment est}, the $\mathcal{K}_\infty$ function $\widehat\gamma$, the concavity of $\widehat\gamma_p$ in \eqref{supplement}, the Jensen inequality \cite{oksendal}, equation \eqref{eta2}, the inequalities (\ref{supplement}) and (\ref{b06}), and triangle inequality, we obtain
\begin{small}
\begin{align}\nonumber
\mathbb{E}[V_{p}(x'_{\tau},\ol{H}_{\ol{\params}}(x'_{{\ol{\params}}},i'))]&=\mathbb{E}[V_{p}(x'_\tau,\xi_{\ol{H}_{\ol{\params}}(x_{{\ol{\params}}},i)p}(\tau))+V_{p}(x'_{\tau},\ol{H}_{\ol{\params}}(x'_{{\ol{\params}}},i'))-V_{p}(x'_\tau,\xi_{\ol{H}_{\ol{\params}}(x_{{\ol{\params}}},i)p}(\tau))]\\ \notag
&=  \mathbb{E}[V_{p}(x'_\tau,\xi_{\ol{H}_{\ol{\params}}(x_{{\ol{\params}}},i)p}(\tau))]+\mathbb{E}[V_{p}(x'_{\tau},\ol{H}_{\ol{\params}}(x'_{{\ol{\params}}},i'))-V_{p}(x'_\tau,\xi_{\ol{H}_{\ol{\params}}(x_{{\ol{\params}}},i)p}(\tau))]\\\notag&\leq\mathsf{e}^{-\kappa\tau}\delta_i+\mathbb{E}[\widehat\gamma_{p}(\Vert\xi_{\ol{H}_{\ol{\params}}(x_{{\ol{\params}}},i)p}(\tau)-\ol{H}_{\ol{\params}}(x'_{{\ol{\params}}},i')\Vert)]\leq\mathsf{e}^{-\kappa\tau}\delta_i+\widehat\gamma_{p}(\mathbb{E}[\Vert\xi_{\ol{H}_{\ol{\params}}(x_{{\ol{\params}}},i)p}(\tau)-\ol{H}_{\ol{\params}}(x'_{{\ol{\params}}},i')\Vert])\\\notag
&\leq\mathsf{e}^{-\kappa\tau}\delta_i+\widehat\gamma(\mathbb{E}[\Vert\xi_{\ol{H}_{\ol{\params}}(x_{{\ol{\params}}},i)p}(\tau)-\ol{\xi}_{\ol{H}_{\ol{\params}}(x_{{\ol{\params}}},i)p}(\tau)+\ol{\xi}_{\ol{H}_{\ol{\params}}(x_{{\ol{\params}}},i)p}(\tau)-\ol{H}_{\ol{\params}}(x'_{{\ol{\params}}},i')\Vert])\\\notag
&\leq\mathsf{e}^{-\kappa\tau}\delta_i+\widehat\gamma(\mathbb{E}[\Vert\xi_{\ol{H}_{\ol{\params}}(x_{{\ol{\params}}},i)p}(\tau)-\ol{\xi}_{\ol{H}_{\ol{\params}}(x_{{\ol{\params}}},i)p}(\tau)\Vert]+\Vert\ol{\xi}_{\ol{H}_{\ol{\params}}(x_{{\ol{\params}}},i)p}(\tau)-\ol{H}_{\ol{\params}}(x'_{{\ol{\params}}},i')\Vert)\\\label{b08}
&\leq\mathsf{e}^{-\kappa\tau}\delta_i+\widehat\gamma((h_{x_s}((N+1)\tau))^{\frac{1}{q}}+\eta)=\delta_{i+1}.
\end{align}
\end{small}
We now examine three separate cases:
\begin{itemize}
\item If $i<\widehat{N}-1$, then $p'=p$, and $i'=i+1$; from (\ref{b08}), \begin{small}$\EE\left[V_p\left(x'_{\tau},\ol{H}_{\ol{\params}}\left(x'_{{\ol{\params}}},i'\right)\right)\right]\leq\delta_{i+1}$\end{small}, we conclude that \begin{small}$(x'_{\tau},p,i+1,x'_{{\ol{\params}}},i+1)\in R$\end{small}; 

\item If $i=\widehat{N}-1$, and $p'=p$, then $i'=\widehat{N}-1$; from (\ref{b08}), \begin{small}$\EE\left[V_p\left(x'_{\tau},\ol{H}_{\ol{\params}}\left(x'_{{\ol{\params}}},i'\right)\right)\right]\leq\delta_{\widehat{N}}\leq\delta_{\widehat{N}-1}$\end{small}, we conclude that \begin{small}$(x'_{\tau},p,\widehat{N}-1,x'_{{\ol{\params}}},\widehat{N}-1)\in R$\end{small}; 

\item If $i=\widehat{N}-1$, and $p'\neq{p}$, then $i'=0$; from (\ref{Ntozero}) and (\ref{b08}), \begin{small}$\EE\left[V_p\left(x'_{\tau},\ol{H}_{\ol{\params}}\left(x'_{{\ol{\params}}},i'\right)\right)\right]\leq\delta_{\widehat{N}}\leq\delta_0/\mu$\end{small}. From (\ref{eq0}), it follows that \begin{small}$\EE\left[V_{p'}(x'_{\tau},\ol{H}_{\ol{\params}}\left(x'_{{\ol{\params}}},i'\right))\right]\leq\mu \EE\left[V_p\left(x'_{\tau},\ol{H}_{\ol{\params}}\left(x'_{{\ol{\params}}},i'\right)\right)\right]\leq\delta_0$\end{small}. Hence, \begin{small}$(x'_{\tau},p',0,x'_{{\ol{\params}}},0)\in R$\end{small}.
\end{itemize}
Therefore, we conclude that condition (ii) in Definition \ref{APSR} holds. In a similar way, we can prove that that $R^{-1}$ is an
$\varepsilon$-approximate simulation relation from $\ol{S}_{{\ol{\params}}}(\Sigma_{\tau_d})$ to $S_{\tau}(\Sigma_{\tau_d})$ implying that $R$ is an $\varepsilon$-approximate bisimulation relation between $\ol{S}_{{\ol{\params}}}(\Sigma_{\tau_d})$ and $S_\tau(\Sigma_{\tau_d})$.
\end{proof}

Note that one can use any over approximation of $\ol\eta$ such as the one in \eqref{upper_bound5} instead of $\ol\eta$ in condition \eqref{bisim_cond_mul0}. By choosing $N$ sufficiently large, one can enforce $h_{x_s}((N+1)\tau)$ and $\ol\eta$ to be sufficiently small. Hence, it can be readily seen that for a given precision $\varepsilon$,
there always exists a large value of $N$, such that the condition in (\ref{bisim_cond_mul0}) is satisfied.

The next theorem provides a result that is similar to the one of Theorem \ref{main_theorem33}, but by using the symbolic model $S_{{\ol{\params}}}(\Sigma_{\tau_d})$.

\begin{theorem}\label{main_theorem22}
Consider a stochastic switched system $\Sigma_{\tau_d}$. Let us assume that for any $p\in\mathsf{P}$, there exists a $\delta$-GAS-M$_q$ Lyapunov function $V_p$ for subsystem $\Sigma_{\tau_d,p}$. Moreover, assume that (\ref{eq0}) holds for some $\mu\geq1$. Let $\widehat\eta$ be given by \eqref{eta3}. If $\tau_d>\log{\mu}/\kappa$, for any $\varepsilon\in\R^+$, 
and any triple $\ol{\mathsf{q}}=\left(\tau,N,x_s\right)$ of parameters satisfying
\begin{small}
\begin{align}
\label{bisim_cond_mul10}
\widehat\gamma\left(\widehat\eta\right)&\leq\frac{\frac{1}{\mu}-\mathsf{e}^{-\kappa\tau_d}}{1-\mathsf{e}^{-\kappa\tau_d}}\left(1-\mathsf{e}^{-\kappa\tau}\right)\underline\alpha\left(\varepsilon^q\right),
\end{align}
\end{small}there exists an $\varepsilon$-approximate bisimulation relation $R$ between ${S}_{\ol{\mathsf{q}}}(\Sigma_{\tau_d})$ and $S_{\tau}(\Sigma_{\tau_d})$ as the following:\\ $\left(x_{\tau},p_1,i_1,x_{{\ol{\params}}},i_2\right)\in R$, where $x_{\ol{\params}}=(\ol{p}_1,\ldots,\ol{p}_N)$, if and only if $p_1=\ol{p}_N=p$, $i_1=i_2=i$, and \begin{small}$$\mathbb{E}[V_p(H_{\tau}(x_{\tau},p_1,i_1),{H}_{{\ol{\params}}}(x_{{\ol{\params}}},i_2))]=\mathbb{E}[V_p(x_{\tau},\xi_{x_sx_{\ol{\params}}}(N\tau))]\leq\delta_i,$$\end{small}where $\delta_0,\ldots,\delta_{\widehat{N}-1}$ are given recursively by \begin{small}$\delta_0=\ul\alpha\left(\varepsilon^q\right),~\delta_{i+1}=\mathsf{e}^{-\kappa\tau}\delta_i+\widehat\gamma\left(\widehat\eta\right)$\end{small}.
\end{theorem} 

\begin{proof}
The proof is similar to the one of Theorem \ref{main_theorem33}.
\end{proof}

Note that one can also use any over approximation of $\widehat\eta$ such as the one in \eqref{upper_bound6} instead of $\widehat\eta$ in condition \eqref{bisim_cond_mul10}. Finally, we establish the results on the existence of symbolic model $\ol{S}_{\ol{\params}}(\Sigma_{\tau_d})$ (resp. $S_{\ol{\params}}(\Sigma_{\tau_d})$) such that \mbox{$\ol{S}_{\ol{\params}}(\Sigma_{\tau_d})\cong_{\mathcal{S}}^{\varepsilon}S_\tau(\Sigma_{\tau_d})$} (resp. \mbox{$S_{\ol{\params}}(\Sigma_{\tau_d})\cong_{\mathcal{S}}^{\varepsilon}S_\tau(\Sigma_{\tau_d})$}).

\begin{theorem}\label{main_theorem55}
Consider the result in Theorem \ref{main_theorem33}. If we choose: \begin{small}
\begin{align}\notag
X_{\tau0}=\big\{&\left(x,p,i\right)\,\,|\,\,x\in\R^n,\left\Vert x-\ol{H}_{\ol{\params}}(x_{{\ol{\params}}0},i)\right\Vert\leq\left(\ol\alpha_p^{-1}\left(\delta_i\right)\right)^{\frac{1}{q}},p=p_N,\notag\forall (p_1,\ldots,p_N,i)\in X_{{\ol{\params}}0}\big\},
\end{align}
\end{small}then we have \mbox{$\ol{S}_{\ol{\params}}(\Sigma_{\tau_d})\cong_{\mathcal{S}}^{\varepsilon}S_\tau(\Sigma_{\tau_d})$}.
\end{theorem}

\begin{proof}
We start by proving that \mbox{$S_{\tau}(\Sigma_{\tau_d})\preceq^{\varepsilon}_\mathcal{S}\ol{S}_{{\ol{\params}}}(\Sigma_{\tau_d})$}. For every $\left(x_{\tau 0},p,i\right)\in{X_{\tau 0}}$, there always exists \mbox{$\left(x_{{\ol{\params}} 0},i\right)\in{X}_{{\ol{\params}} 0}$}, where $x_{{\ol{\params}}0}=(p_1,\ldots,p_N)$, such that $p=p_N$ and \begin{small}$\left\Vert{x_{\tau0}}-\ol{H}_{\ol{\params}}(x_{{\ol{\params}}0},i)\right\Vert\leq\left(\ol\alpha_p^{-1}\left(\delta_i\right)\right)^{\frac{1}{q}}$\end{small}. Then,
\begin{footnotesize}
\begin{align}\nonumber
\mathbb{E}\left[V_p\left({x_{\tau0}},\ol{H}_{\ol{\params}}(x_{{\ol{\params}}0},i)\right)\right]&=V_p\left({x_{\tau0}},\ol{H}_{\ol{\params}}(x_{{\ol{\params}}0},i)\right)\leq\overline\alpha_p(\Vert x_{\tau0}-\ol{H}_{\ol{\params}}(x_{{\ol{\params}}0},i)\Vert^q)\leq\delta_i,
\end{align}
\end{footnotesize}since $\overline\alpha_p$ is a $\mathcal{K}_\infty$ function.
Hence, \mbox{$\left(x_{\tau0},p,i,x_{{\ol{\params}}0},i\right)\in{R}$} implying that $S_{\tau}(\Sigma_{\tau_d})\preceq^{\varepsilon}_\mathcal{S}\ol{S}_{{\ol{\params}}}(\Sigma_{\tau_d})$. In a similar way, we can show that \mbox{$\ol{S}_{{\ol{\params}}}(\Sigma_{\tau_d})\preceq^{\varepsilon}_{\mathcal{S}}S_{\tau}(\Sigma_{\tau_d})$}, equipped with the relation $R^{-1}$, which completes the proof.
\end{proof}

The next theorem provides a similar result as the one of Theorem \ref{main_theorem55}, but by using the symbolic model $S_{\ol{\params}}(\Sigma)$.

\begin{theorem}\label{main_theorem66}
Consider the results in Theorem \ref{main_theorem22}. If we choose: \begin{small}
\begin{align}\notag
X_{\tau0}=\big\{&\left(a,p,i\right)\,\,|\,\,a\in\mathcal{X}_0,\left(\EE\left[\left\Vert a-H_{\ol{\params}}(x_{{\ol{\params}}0},i)\right\Vert^q\right]\right)^{\frac{1}{q}}\leq\left(\ol\alpha_p^{-1}\left(\delta_i\right)\right)^{\frac{1}{q}}, p=p_N, \forall (p_1,\ldots,p_N,i)\in X_{{\ol{\params}}0}\big\},
\end{align}
\end{small}then we have \mbox{$S_{\ol{\params}}(\Sigma_{\tau_d})\cong_{\mathcal{S}}^{\varepsilon}S_\tau(\Sigma_{\tau_d})$}.
\end{theorem}

\begin{proof}
The proof is similar to the one of Theorem \ref{main_theorem55}.
\end{proof}

\begin{remark}\label{remark1}
The symbolic model $S_{\ol{\params}}(\Sigma)$ (resp. $S_{\ol{\params}}(\Sigma_{\tau_d})$), computed by using the parameter $\ol{\mathsf{q}}$ provided in Theorem \ref{main_theorem3} (resp. Theorem \ref{main_theorem22}), has fewer (or at most equal number of) states than the symbolic model $\ol{S}_{\ol{\params}}(\Sigma)$ (resp. $\ol{S}_{\ol{\params}}(\Sigma_{\tau_d})$), computed by using the parameter ${\ol{\params}}$ provided in Theorem \ref{main_theorem} (resp. Theorem \ref{main_theorem33}) while having the same precision. However, the symbolic models $S_{\ol{\params}}(\Sigma)$ and $S_{\ol{\params}}(\Sigma_{\tau_d})$ have states with probabilistic output values, rather than non-probabilistic ones which makes the control synthesis over them more involved.
\end{remark}

\begin{remark}
The control synthesis over $\ol{S}_{\ol{\params}}(\Sigma)$ (resp. $\ol{S}_{\ol{\params}}(\Sigma_{\tau_d})$) is simple as the outputs are non-probabilistic points. For ${S}_{\ol{\params}}(\Sigma)$ (resp. ${S}_{\ol{\params}}(\Sigma_{\tau_d})$) it is less intuitive and more involved. We refer the interested readers to \cite[Subsection 5.3]{majid10} explaining how one can synthesize controllers over finite metric systems with random output values.
\end{remark}

\subsection{Comparison between the two proposed approaches}

Note that given any precision $\varepsilon$ and sampling time $\tau$, one can always use the results proposed in Theorems \ref{main_theorem5} and \ref{main_theorem55} to construct symbolic models $\ol{S}_\params(\Sigma)$ and $\ol{S}_\params(\Sigma_{\tau_d})$, respectively, that are $\varepsilon$-approximately bisimilar to $S_\tau(\Sigma)$ and $S_\tau(\Sigma_{\tau_d})$, respectively. However, the results proposed in Theorems \ref{main_theorem1} and \ref{main_theorem2} cannot be applied for any sampling time $\tau$ if the precision $\varepsilon$ is lower than the thresholds introduced in inequalities \eqref{lower_bound} and \eqref{lower_bound_mul}, respectively (cf. the first case study). Furthermore, while the results in Theorems \ref{main_theorem1} and \ref{main_theorem2} only provide symbolic models with non-probabilistic output values, the ones in Theorems \ref{main_theorem6} and \ref{main_theorem66} provide symbolic models with probabilistic output values as well which can result in less conservative symbolic models (cf. Remark \ref{remark1} and the first case study).

One can compare the results provided in Theorems \ref{main_theorem5} and \ref{main_theorem55} with the results provided in Theorems \ref{main_theorem1} and \ref{main_theorem2}, respectively, in terms of the sizes of the symbolic models. One can readily verify that the precision of the symbolic model $\ol{S}_{\ol{\params}}(\Sigma)$ (resp. $\ol{S}_{\ol{\params}}(\Sigma_{\tau_d})$) and the one $S_\params(\Sigma)$ (resp. $S_\params(\Sigma_{\tau_d})$) is approximately the same as long as the state space quantisation parameter $\eta$ is equal to the parameter $\ol\eta$ in \eqref{eta} (resp. in \eqref{eta2}), i.e. \begin{small}$\eta\leq\left(\ul\alpha^{-1}\left(\mathsf{e}^{-\kappa N\tau}\ol{\eta}_0\right)\right)^{1/q}$\end{small} (resp. \begin{small}$\eta\leq\left(\ul\alpha^{-1}\left(\mathsf{e}^{-(\kappa-\log\mu/{\tau_d})N\tau}\widehat\eta_0\right)\right)^{1/q}$\end{small}), where \begin{small}$\ol{\eta}_0=\max_{p\in\mathsf{P}}V\left(\ol\xi_{x_sp}(\tau),x_s\right)$\end{small} (resp. \begin{small}$\widehat\eta_0=\max_{p,p'\in\mathsf{P}}V_{p'}\left(\ol\xi_{x_sp}(\tau),x_s\right)$\end{small}). The reason their precisions are approximately (not exactly) the same is because we use $(h_{x_s}((N+1)\tau))^{1/q}$ in conditions \eqref{bisim_cond11} and \eqref{bisim_cond_mul0} rather than $(h_{[X_{\tau0}]_\eta}(\tau))^{1/q}$ (resp. $(h_{[X_0]_\eta}(\tau))^{1/q}$) that is being used in condition\eqref{bisim_cond} (resp.\eqref{bisim_cond_mul}). By assuming that $(h_{x_s}((N+1)\tau))^{1/q}$ and $(h_{[X_{\tau0}]_\eta}(\tau))^{1/q}$ (resp. $(h_{[X_0]_\eta}(\tau))^{1/q}$) are much smaller than $\overline{\eta}$ and $\eta$, respectively, or $h_{x_s}((N+1)\tau)\approx h_{[X_{\tau0}]_\eta}(\tau)$ (resp. $h_{x_s}((N+1)\tau)\approx h_{[X_0]_\eta}(\tau)$), the precisions are the same.

The number of states of the proposed symbolic models $\ol{S}_{\ol{\params}}(\Sigma)$ and $\ol{S}_{\ol{\params}}(\Sigma_{\tau_d})$ are $m^N$ and $m^{N}\times\widehat{N}$, respectively. Assume that we are interested in the dynamics of $\Sigma$ (resp. $\Sigma_{\tau_d}$) on a compact set $\mathsf{D}\subset\R^n$. Since the set of states of the proposed symbolic models $S_\params(\Sigma)$ and $S_\params(\Sigma_{\tau_d})$ are $\left[\mathsf{D}\right]_{\eta}$ and $\left[\mathsf{D}\right]_{\eta}\times\mathsf{P}\times\{0,\ldots,\widehat{N}-1\}$, respectively, their sizes are $\left\vert\left[\mathsf{D}\right]_{\eta}\right\vert=\frac{K}{\eta^n}$ and $\frac{K}{\eta^n}\times{m}\times\widehat{N}$, respectively, where $K$ is a positive constant proportional to the volume of $\mathsf{D}$. Hence, it is more convenient to use the proposed symbolic models $\ol{S}_{\ol{\params}}(\Sigma)$ and $\ol{S}_{\ol{\params}}(\Sigma_{\tau_d})$ rather than the ones $S_\params(\Sigma)$ and $S_\params(\Sigma_{\tau_d})$, respectively, as long as:
\begin{small}
\begin{align}\nonumber
&m^N\leq\frac{K}{\left(\ul\alpha^{-1}\left(\mathsf{e}^{-\kappa N\tau}\ol\eta_0\right)\right)^{n/q}}~~\text{and}~~m^{N-1}\leq\frac{K}{\left(\ul\alpha^{-1}\left(\mathsf{e}^{-(\kappa-\log\mu/{\tau_d}) N\tau}\widehat\eta_0\right)\right)^{n/q}},
\end{align}
\end{small}respectively. Without loss of generality, one can assume that $\ul\alpha({r})=r$ for any $r\in\R_0^+$. Hence, for sufficiently large value of $N$, it is more convenient to use the proposed symbolic models $\ol{S}_{\ol{\params}}(\Sigma)$ and $\ol{S}_{\ol{\params}}(\Sigma_{\tau_d})$ in comparison with the ones $S_\params(\Sigma)$ and $S_\params(\Sigma_{\tau_d})$, respectively, as long as:
\begin{small}
\begin{align}\label{criterion}
m\mathsf{e}^{\frac{-\kappa\tau n}{q}}\leq1,~~\text{and}~~m\mathsf{e}^{\frac{-(\kappa-\log\mu/{\tau_d})\tau n}{q}}\leq1,
\end{align}
\end{small}respectively.

\section{Examples}

\subsection{Room temperature control (common Lyapunov function)}
Consider the stochastic switched system $\Sigma$ which is a simple thermal model of a six-room building as depicted schematically in Figure \ref{fig1} and described by the following stochastic differential equations:
\begin{footnotesize}
\begin{align}\label{room}
\hspace{-1.5mm}\begin{array}{l}
                  \diff{\xi}_1=\big(\alpha_{21}\left(\xi_2-\xi_1\right)+\alpha_{31}\left(\xi_3-\xi_1\right)+\alpha_{51}\left(\xi_5-\xi_1\right)+\alpha_{e1}\left(T_e-\xi_1\right)+\alpha_{f1}\left(T_{f1}-\xi_1\right)\delta_{p2}\big)\diff{t}+\left(\sigma_{1,1}\delta_{p1}+(1-\delta_{p1})\sigma_1\right)\xi_1\diff{W^1_t},\\ 
\diff{\xi}_2=\left(\alpha_{12}\left(\xi_1-\xi_2\right)+\alpha_{42}\left(\xi_4-\xi_2\right)+\alpha_{e2}\left(T_e-\xi_2\right)\right)\diff{t}+\left(\sigma_{2,1}\delta_{p1}+(1-\delta_{p1})\sigma_2\right)\xi_2\diff{W^2_t},\\
\diff{\xi}_3=\left(\alpha_{13}\left(\xi_1-\xi_3\right)+\alpha_{43}\left(\xi_4-\xi_3\right)+\alpha_{e3}\left(T_e-\xi_3\right)\right)\diff{t}+\left(\sigma_{3,1}\delta_{p1}+(1-\delta_{p1})\sigma_3\right)\xi_3\diff{W^3_t},\\
\diff{\xi}_4=\big(\alpha_{24}\left(\xi_2-\xi_4\right)+\alpha_{34}\left(\xi_3-\xi_4\right)+\alpha_{64}\left(\xi_6-\xi_4\right)+\alpha_{e4}\left(T_e-\xi_4\right)+\alpha_{f4}\left(T_{f4}-\xi_4\right)\delta_{p3}\big)\diff{t}+\left(\sigma_{4,1}\delta_{p1}+(1-\delta_{p1})\sigma_4\right)\xi_4\diff{W^4_t},\\ 
\diff{\xi}_5=\left(\alpha_{15}\left(\xi_1-\xi_5\right)+\alpha_{e5}\left(T_e-\xi_5\right)\right)\diff{t}+\left(\sigma_{5,1}\delta_{p1}+(1-\delta_{p1})\sigma_5\right)\xi_5\diff{W^5_t},\\
\diff{\xi}_6=\left(\alpha_{46}\left(\xi_4-\xi_6\right)+\alpha_{e6}\left(T_e-\xi_6\right)\right)\diff{t}+\left(\sigma_{6,1}\delta_{p1}+(1-\delta_{p1})\sigma_6\right)\xi_6\diff{W^6_t},\\                  
\end{array}
\end{align}
\end{footnotesize}where the terms $W_t^i$, $i=1,\ldots,6$, denote the standard Brownian motion and $\delta_{pi}=1$ if $i=p$ and $\delta_{pi}=0$ otherwise. 

\begin{figure}[h]
\begin{center}
\includegraphics[width=8cm]{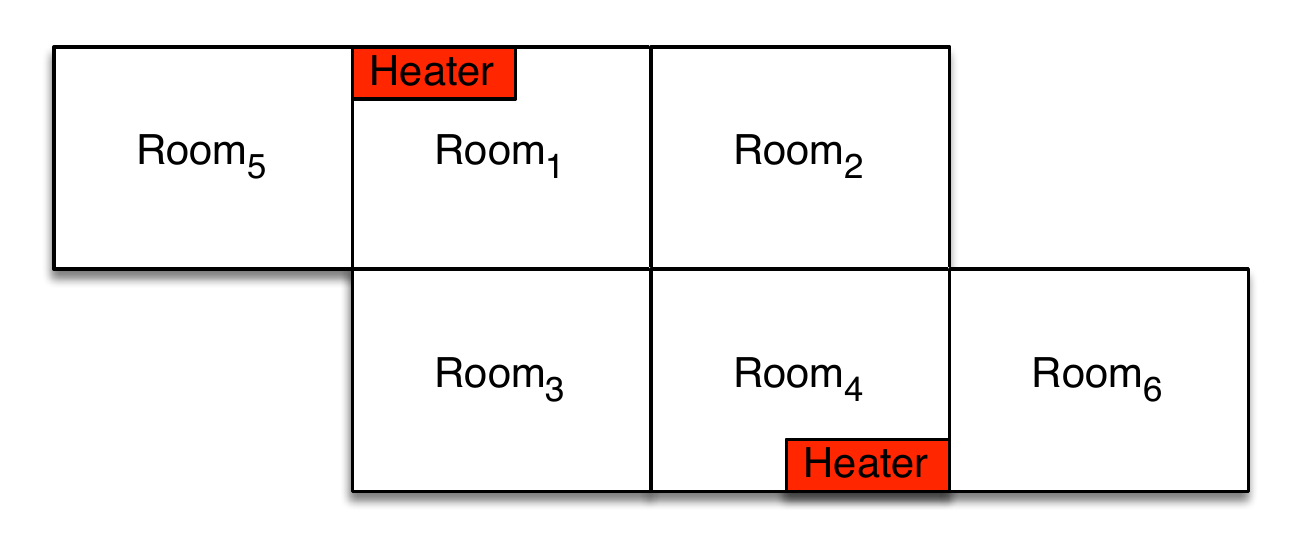}
\end{center}
\caption{A schematic of the six-room building.}
\label{fig1}
\end{figure}

Note that $\xi_i$, $i=1,\ldots,6$, denotes the temperature in each room, 
$T_e=10$ (degrees Celsius) is the external temperature, and $T_{f1}=T_{f4}=100$ are the temperatures of two heaters\footnote{Here, we assume that at most one heater is on at each instant of time.} that both can be switched off ($p=1$), 1st heater ($T_{f1}$) on and the 2nd one ($T_{f4}$) off ($p=2$), or vice versa ($p=3$). 
The drifts $f_p$ and diffusion terms $g_p$, $p = 1,2,3$, can be simply written out of \eqref{room} and are affine and linear, respectively. 
The parameters of the drifts are chosen as follows: $\alpha_{21}=\alpha_{12}=\alpha_{13}=\alpha_{31}=\alpha_{42}=\alpha_{24}=\alpha_{34}=\alpha_{43}=\alpha_{15}=\alpha_{51}=\alpha_{46}=\alpha_{64}=5\times10^{-2}$, $\alpha_{e1}=\alpha_{e4}=5\times10^{-3}$, $\alpha_{e2}=\alpha_{e3}=\alpha_{e5}=\alpha_{e6}=3.3\times10^{-3}$, and $\alpha_{f1}=\alpha_{f4}=3.6\times10^{-3}$. 
The noise parameters are chosen as $\sigma_{i,1}=0.002$ and $\sigma_i=0.003$, for $i=1,\ldots,6$. 

It can be readily verified that the function $V(x,x')=\sqrt{(x-x')^T(x-x')}$ satisfies the LMI condition (9) in \cite{majid9} with $q=1$, $P_p=I_6$, and $\widehat\kappa_p=0.0076$, for any $p\in\{1,2,3\}$. Hence, $V$ is a common $\delta$-GAS-M$_1$ Lyapunov function for $\Sigma$, satisfying conditions (i)-(iii) in Definition \ref{delta_SGAS_Lya} with $q=1$, $\ul\alpha_p({r})=\ol\alpha_p({r})=r$, $\forall r\in\R_0^+$, and $\kappa_p=0.0038$, for any $p\in\{1,2,3\}$. Using the results of Theorem \ref{theorem2}, one gets that function $\beta(r,s)=\mathsf{e}^{-\kappa_p{s}}r$ satisfies property \eqref{delta_SGUAS} for $\Sigma$.

For a \emph{source state}\footnote{Note that here we computed the source state as $x_s=\arg\min_{x\in\R^n}\max_{p\in\mathsf{P}}V(\ol\xi_{xp}(\tau),x)$ in order to have the smallest upper bound for $\ol\eta$ as in \eqref{upper_bound}.} $x_s=[18,17.72,17.72,18,17.46,17.46]^T$, a given sampling time $\tau=30$ time units, and a selected precision $\varepsilon=1$, the parameter $N$ for $\ol{S}_{\ol{\params}}(\Sigma)$, based on inequality (\ref{bisim_cond11}) in Theorem \ref{main_theorem}, is obtained as 13 and one gets $\ol\eta\leq0.1144$, where $\ol\eta$ is given in \eqref{eta}. Therefore, the resulting cardinality of the set of states for $\ol{S}_{\ol{\params}}(\Sigma)$ is $3^{13}=1594323$.

Now, consider that the objective is to design a control policy forcing the trajectories of $\Sigma$, starting from the initial condition $x_0=[11.7,11.7,11.7,11.7,11.7,11.7]^T$, to reach the region $\mathsf{D}=[19~22]^6$ in finite time and remain there forever. 
This objective can be encoded via the LTL specification $\Diamond\Box\mathsf{D}$. 

In Figure \ref{fig2}, we show several realizations of the trajectory $\xi_{x_0\upsilon}$ stemming from initial condition $x_0$ (top panels), 
as well as the corresponding evolution of synthesized switching signal $\upsilon$ (bottom panel). 
Furthermore, in Figure \ref{fig3}, we show the average value over 10000 experiments of the distance in time of the solution process $\xi_{x_0\upsilon}$ to the set $\mathsf{D}$, namely $\left\Vert\xi_{x_0\upsilon}(t)\right\Vert_\mathsf{D}$, where the point-to-set distance is defined as $\Vert{x}\Vert_\mathsf{D}=\inf_{d\in{\mathsf{D}}}\Vert x-d\Vert$.   

\begin{figure}[h]
\begin{center}
\includegraphics[width=12cm]{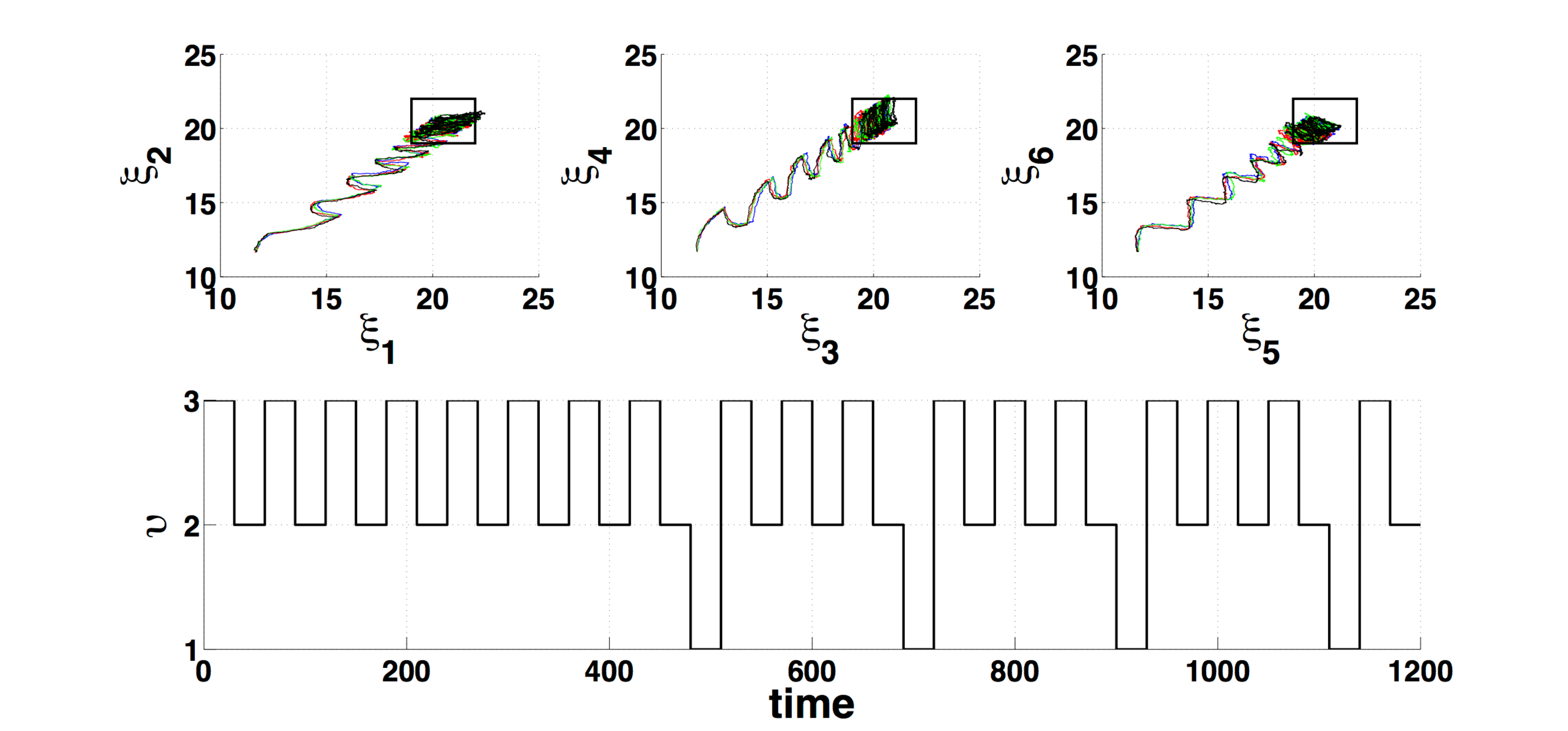}
\end{center}
\caption{A few realizations of the solution process $\xi_{x_0\upsilon}$ (top panel) and the corresponding evolution of the obtained switching signal $\upsilon$ (bottom panel).}
\label{fig2}
\end{figure}

\begin{figure}[h]
\begin{center}
\includegraphics[width=10cm]{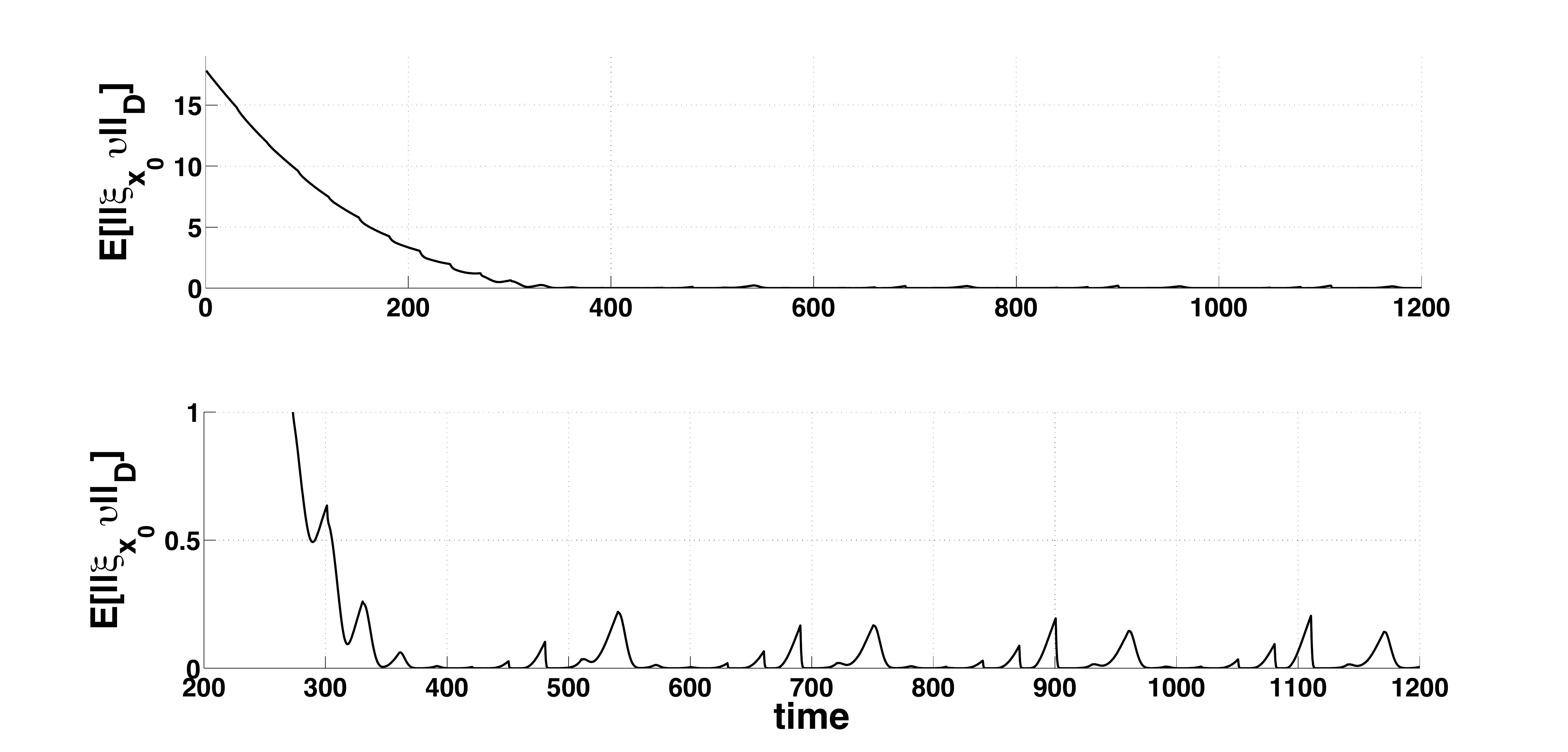}
\end{center}
\caption{The average values (over 10000 experiments) of the distance of the solution process $\xi_{x_0\upsilon}$ to the set $\mathsf{D}$ in different vertical scales.}
\label{fig3}
\end{figure}

To compute exactly the size of the symbolic model, proposed in Theorem \ref{main_theorem1}, we consider the dynamics of $\Sigma$ over the subset $\mathsf{W}=[11.7~22]^6\subset\R^6$. Note that using the sampling time $\tau=30$, the results in Theorem \ref{main_theorem1} cannot be applied because the precision $\varepsilon$ has to be lower bounded by $2.7$ as in inequality \eqref{lower_bound}. Using a bigger precision $\varepsilon=2.8$ than the one here, the same sampling time $\tau=30$ as the one here, and the inequalities \eqref{bisim_cond1} and \eqref{bisim_cond}, we obtain the state space quantization parameter as $\eta\leq0.02$. Therefore, if one uses $\eta=0.02$, the cardinality of the state set of the symbolic model $S_\params(\Sigma)$ is equal to $\left(\frac{22-11.7}{0.02}\right)^6=1.8657\times10^{16}$ which is much higher than the one of $\ol S_{\ol{\params}}(\Sigma)$, i.e. $1594323$, while having even larger precision.

\begin{remark}
By considering the dynamics of $\Sigma$ over the set $\mathsf{W}$, at least $1-10^{-5}$ confidence level, and using Hoeffding's inequality \cite{hoeffding}, one can verify that the number of samples should be at least $74152$ to empirically compute the upper bound of $\widehat\eta$ in \eqref{upper_bound4}. We compute $\widehat\eta\leq0.1208$ when $x_s=[18~17.72~17.72~18~17.46~17.46]^T$, $N=13$, and $\tau=30$. Using the results in Theorem \ref{main_theorem3} and the same parameters $\ol\params$ as the ones in $\ol{S}_{\ol\params}(\Sigma)$, one obtains $\varepsilon=0.6$ in \eqref{bisim_cond3}. Therefore, $S_{\ol\params}(\Sigma)$, with confidence at least $1-10^{-5}$, provides a less conservative precision than $\ol{S}_{\ol\params}(\Sigma)$, while having the same size as $\ol{S}_{\ol\params}(\Sigma)$. 
\end{remark}

\begin{remark}
Another advantage of using the 2nd approach in comparison with the 1st one is that one can construct only a relevant part of the abstraction given an initial condition and the specification which was the case in this example. 
\end{remark}
\subsection{Multiple Lyapunov functions}
 Consider the following stochastic switched system borrowed from \cite{girard2} and additionally affected by noise:  
\begin{small}
\begin{align}\nonumber
\Sigma:
\left\{
\begin{array}{l}
 \diff\xi_1=\left(-0.25\xi_1+p\xi_2+(-1)^p0.25\right)\diff{t}+0.01\xi_1\diff{W_t^1},\\
 \diff\xi_2=\left(\left(p-3\right)\xi_1-0.25\xi_2+(-1)^p\left(3-p\right)\right)\diff{t}+0.01\xi_2\diff{W_t^2},
 \end{array}\right.
\end{align}
\end{small}where $p=1,2$. 
The noise-free version of $\Sigma$ is endowed with stable subsystems, 
however it can globally exhibit unstable behaviors for some switching signals \cite{girard2}. 
Similarly, $\Sigma$ does not admit a common $\delta$-GAS-M$_q$ Lyapunov function. 
We are left with the option of seeking for multiple Lyapunov functions. 
It can be indeed shown that each subsystem $\Sigma_p$ admits a $\delta$-GAS-M$_1$ Lyapunov function of the form \begin{small}$V_p(x_1,x_2)=\sqrt{(x_1-x_2)^TP_p(x_1-x_2)}$\end{small}, with \begin{small}$P_1=\left[ {\begin{array}{cc}
2&0\\
0&1\\
 \end{array} } \right]$\end{small} and \begin{small}$P_2=\left[ {\begin{array}{cc}
1&0\\
0&2\\
 \end{array} } \right].$\end{small}
These $\delta$-GAS-M$_1$ Lyapunov functions have the following characteristics: 
$\ul\alpha({r})=r$, $\ol\alpha({r})=2r$, $\kappa=0.2498$. 
Note that $V^2_p(x_1,x_2)$ is also a $\delta$-GAS-M$_2$ Lyapunov function for $\Sigma_p$, where $p\in\{1,2\}$, satisfying the requirements in Lemma \ref{lem:moment est}. 
Furthermore, the assumptions of Theorem \ref{multiple_lyapunov} hold by choosing a parameter $\mu=\sqrt{2}$ and a dwell time $\tau_d=2>\log{\mu}/\kappa$. 
In conclusion, the stochastic switched system $\Sigma$ is $\delta$-GUAS-M$_1$. 
 
Let us work within the set $\mathsf{D}=[-5,~5]\times[-4,~4]$ of the state space of $\Sigma$. 
For a sampling time $\tau=0.5$, using inequality (\ref{lower_bound_mul}) the precision $\varepsilon$ is lower bounded by $1.07$. 
For a chosen precision $\varepsilon=1.2$, the discretization parameter $\eta$ of $S_{\params}(\Sigma)$, 
obtained from Theorem \ref{main_theorem2}, 
is equal to $0.0083$. 
The resulting number of states in $S_{\params}(\Sigma_{\tau_d})$ is $9310320$, taking 3.4 MB memory space, where the computation of the abstraction $S_\params(\Sigma_{\tau_d})$ has been performed via the software tool \textsf{CoSyMA}~\cite{CoSyMA} on an iMac with CPU $3.5$GHz Intel Core i$7$.
The CPU time needed for computing the abstraction has amounted to $22$ seconds. 
 
Consider the objective to design a controller (switching signal) forcing the first moment of the trajectories of $\Sigma$ to stay within $\mathsf{D}$, 
while always avoiding the set $\mathsf{Z}=[-1.5,1.5]\times[-1,1]$. 
This corresponds to the following LTL specification: $\Box\mathsf{D}\wedge\Box\neg{\mathsf{Z}}$. 
The CPU time needed for synthesizing the controller has amounted to $12.46$ seconds. 
Figure \ref{fig22} displays several realizations of the closed-loop trajectory of $\xi_{x_0\upsilon}$, 
stemming from the initial condition $x_0=(-4,-3.8)$ (left panel), 
as well as the corresponding evolution of the switching signal $\upsilon$ (right panel). 
Furthermore, Figure \ref{fig22} (middle panels) shows the average value (over 10000 experiments) of the distance in time of the solution process $\xi_{x_0\upsilon}$ to the set $\mathsf{D}\backslash{\mathsf{Z}}$, namely $\left\Vert\xi_{x_0\upsilon}(t)\right\Vert_{\mathsf{D}\backslash{\mathsf{Z}}}$. 
Notice that the empirical average distance is significantly lower than the theoretical precision $\varepsilon=1.2$. 

\begin{figure}
\includegraphics[width=13cm]{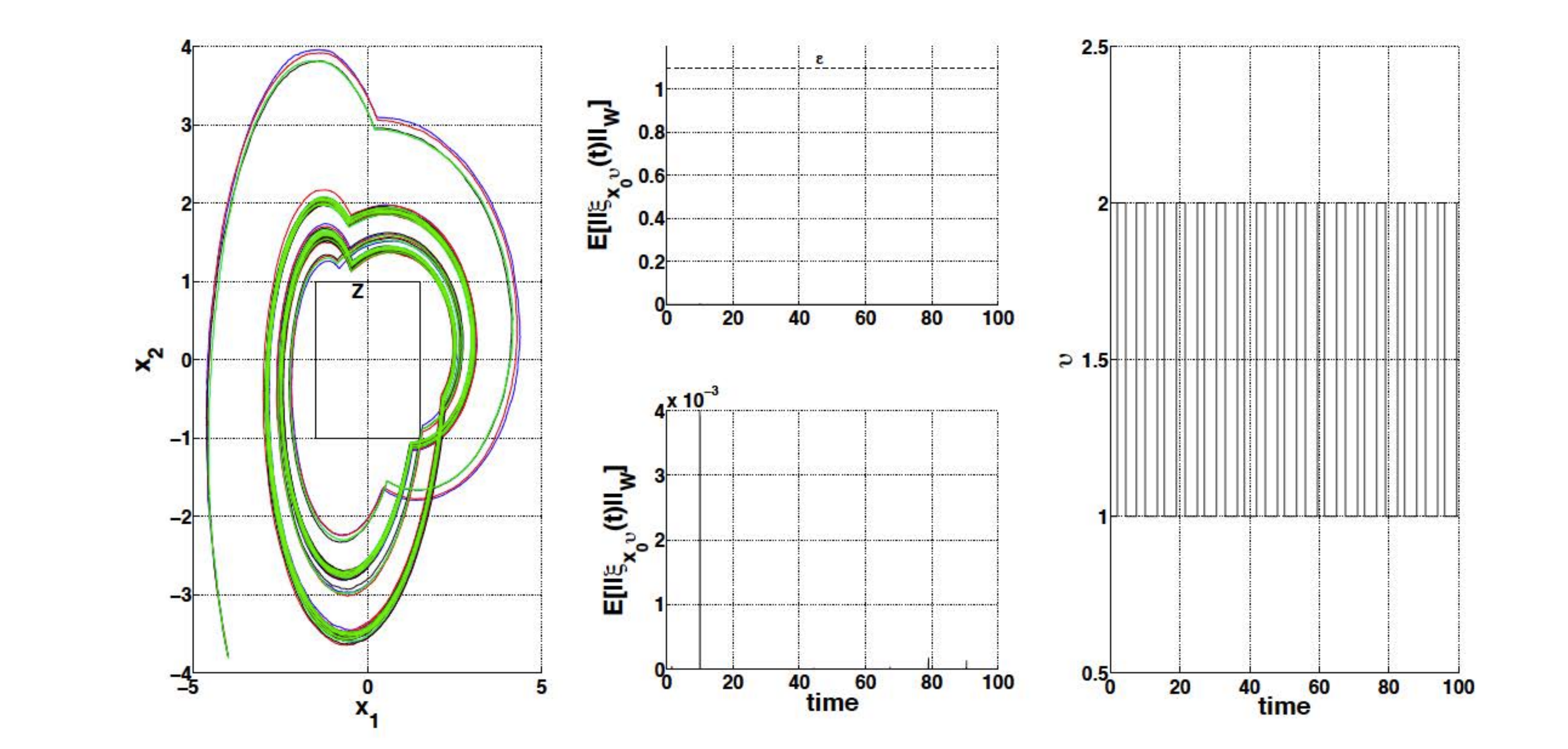}
\caption{Several realizations of the closed-loop trajectory $\xi_{x_0\upsilon}$ with initial condition \mbox{$x_0=(-4,-3.8)$} (left panel). 
Average values (over 10000 experiments) in time of the distance of solution process $\xi_{x_0\upsilon}$ to the set $\mathsf{W}=\mathsf{D}\backslash{\mathsf{Z}}$, in different vertical scales (middle panel). 
Evolution of the synthesized switching signal $\upsilon$ (right panel).}
\label{fig22}
\end{figure}

Note that using the same sampling time $\tau=0.5$, the same precision $\varepsilon=1.2$, and the inequalities \eqref{bisim_cond_mul0} in Theorem \ref{main_theorem33}, we obtain the temporal horizon as $N=22$. Therefore, the cardinality of the state set of the symbolic model $\ol{S}_{\ol{\params}}(\Sigma_{\tau_d})$ is equal to $2^{22}=4194304$ which is roughly half of the one of $S_{\params}(\Sigma_{\tau_d})$, i.e. $9310320$.

\section{Conclusions}
This work has shown that any stochastic switched system $\Sigma$ (resp. $\Sigma_{\tau_d}$), admitting a common (multiple) $\delta$-GAS-M$_q$ Lyapunov function(s), and within a compact set of states, admits an approximately bisimilar symbolic model $S_\params(\Sigma)$ (resp. $S_\params(\Sigma_{\tau_d})$) requiring a space discretization or $S_{\ol{\params}}(\Sigma)/\ol{S}_{\ol{\params}}(\Sigma)$ (resp. $S_{\ol{\params}}(\Sigma_{\tau_d})/\ol{S}_{\ol{\params}}(\Sigma_{\tau_d})$) without any space discretization. Furthermore, we have provided a simple criterion by which one can choose between the two proposed abstraction approaches the most suitable one (based on the size of the abstraction) for a given stochastic switched system.
The constructed symbolic models can be used to synthesize controllers enforcing complex logic specifications, 
expressed via linear temporal logic or as automata on infinite strings. 


\section{Acknowledgements}
The authors would like to thank Ilya Tkachev for fruitful technical discussions. 

\bibliographystyle{plain}
\bibliography{reference}
\end{document}